\documentclass[a4paper,11pt]{article}
\usepackage{hyperref,makecell,amsmath,amssymb,amscd,amsthm,makeidx,txfonts,graphicx,url,bm,relsize}
\usepackage{a4wide,xy}
\usepackage{float,blkarray}
\usepackage{tikz}
\usepackage{tikz-cd}

\let\oldmarginpar\marginpar
\renewcommand\marginpar[1]{\-\oldmarginpar[\raggedleft\footnotesize #1]
{\raggedright\footnotesize #1}}

\makeindex

\xyoption{all}
\input{xypic}

\numberwithin{equation}{section}

\newtheorem{theorem}{Theorem}[subsection]
\newtheorem{proposition}[theorem]{Proposition}

\newtheorem{conjecture}[theorem]{Conjecture}

\newtheorem{lemma}[theorem]{Lemma}

\theoremstyle{remark}
\newtheorem{remark}[theorem]{Remark}
\newtheorem{example}[theorem]{Example}

\theoremstyle{definition}

\def\GL{{\rm GL}}

\def\calP{{\mathcal{P}}}
\def\v{{\bf v}}
\def\u{{\bf u}}
\def\e{{\bf e}}
\def\C{{\bf C}}
\def\Z{{\mathbb{Z}}}
\def\N{{\mathbb{N}}}
\def\muhat{{\boldsymbol \mu}}

\def\tauhat{{\boldsymbol \tau}}
\def\omegahat{{\boldsymbol \omega}}
\def\calU{{\mathcal{U}}}
\def\x{{\bf x}}

\DeclareMathOperator{\Tr}{Tr}

\def\F{\mathbb{F}}

\def\gl{{\rm gl}}

\begin{document}

\title{The Saxl conjecture and the tensor square of unipotent characters of $\GL_n(q)$}

\author{ Emmanuel Letellier
\\ {\it Universit\'e Paris Cit\'e, IMJ-PRG, CNRS}
\\{\tt emmanuel.letellier@imj-prg.fr} \and GyeongHyeon Nam \\ {\it
  Ajou university} \\{\tt ghnam@ajou.ac.kr
 }
 }

\pagestyle{myheadings}

\maketitle

\begin{abstract}We know from \cite{L1} that if for some triple of partitions $(\lambda,\mu,\nu)$ of $n$ the Kronecker coefficient $\langle \chi^\lambda\otimes\chi^\mu,\chi^\nu\rangle$ is non-zero then the corresponding multiplicity $\langle\calU^\lambda\otimes\calU^\mu,\calU^\nu\rangle$ for the unipotent characters of $\GL_n(\F_q)$ is also non-zero. A conjecture of Saxl says that if $\mu$ is a staircase partition, then all irreducible characters of $S_{|\mu|}$ appear non-trivially in the tensor square $\chi^\mu\otimes\chi^\mu$. Therefore the Saxl conjecture implies its analogue for unipotent characters, i.e. all unipotent characters of $\GL_{|\mu|}(\F_q)$ appear non-trivially in the tensor square $\calU^\mu\otimes\calU^\mu$ when $\mu$ is a staircase partition. In this paper we prove the analogue of the Saxl conjecture for unipotent characters. In a second part we describe conjecturally the set of all partitions $\mu$ for which the tensor square $\calU^\mu\otimes\calU^\mu$ contains non-trivially all the unipotent characters of $\GL_{|\mu|}(\F_q)$. 
  \end{abstract}

\section{Introduction}\label{intro}

For a partition $\mu$ of $n$, we let $\chi^\mu$ denote the corresponding irreducible character of the symmetric group $S_n$ and we let $\calU^\mu$ denote the corresponding unipotent character of $\GL_n(\F_q)$ (see \S \ref{unipotent}). If $\mu=(n)$ then $\chi^\mu$ and $\calU^\mu$ are the trivial characters and if $\mu=(1^n)$, then $\chi^\mu$ is the sign character and $\calU^\mu$ is the Steinberg character ${\rm St}$. 
\bigskip

Given a triple of partitions $\muhat=(\mu^1,\mu^2,\mu^3)$ of $n$ we consider

$$
g_\muhat:=\langle\chi^{\mu^1}\otimes\chi^{\mu^2}\otimes\chi^{\mu^3},1\rangle_{S_n},\hspace{.5cm}U_\muhat(q):=\left\langle\calU^{\mu^1}\otimes\calU^{\mu^2}\otimes\calU^{\mu^3},1\right\rangle_{\GL_n(\F_q)}.
$$
The first one is a non-negative integer known as a \emph{Kronecker coefficient} and the second one is a polynomial in $q$. One of the most challenging problem in algebraic combinatorics (going back to Murnaghan in 1938, cf. \cite{Mu})  is to describe combinatorially the set 

$$
\{\muhat=(\mu^1,\mu^2,\mu^3)\,|\, g_\muhat\neq 0\}.
$$
 The  analogous problem for $U_\muhat(q)$ is relatively new. As far as we know, it was first investigated in \cite{HLM}. Since then, substantial progress was made \cite{Lu1}\cite{L1}\cite{L2}\cite{HSTZ}\cite{S}.
 \bigskip
 
 In \cite{L1}, it is shown that the two problems are somehow related, namely it is proved that if $g_\muhat$ is non-zero  then $U_\muhat(q)$ is also non-zero (the converse is false).
\bigskip

A 10 years old conjecture due to Saxl states that if $\mu^1=\mu^2$ is a staircase partition then $g_\muhat$ is non-zero for any partition $\mu^3$. The main theorem of this paper is the following theorem (as predicted by the Saxl conjecture and the results of \cite{L1}), which is proved in \S\ref{s:mainthm}.
\bigskip

\begin{theorem}[Analogue of the Saxl conjecture for unipotent characters] If $\mu^1=\mu^2$ is a staircase partition, then $U_\muhat(q)\neq 0$ for any partition $\mu^3$.
\label{theointro}\end{theorem}
\bigskip

In a second part we are interested in the unipotent characters whose tensor-square contains all the unipotent characters. By the above theorem, the unipotent characters  attached to the staircase partitions do satisfy this property.  By \cite{HSTZ}\cite{L1}, we know that the Steinberg characters ${\rm St}$ also satisfy this property.

We make the following conjecture (see Conjecture \ref{conj}).

\begin{conjecture} Let $\mu=(\mu_1,\mu_2,\dots)$ be a partition of $n$. 

(1) If $\mu_1\leq \lceil n/2\rceil$, then for any partition $\tau$ of $n$ we have

$$
U_{(\mu,\mu,\tau)}(q)\neq 0.
$$

(2) If $\mu_1>\lceil n/2\rceil$, then 

$$
U_{(\mu,\mu,(1^n))}(q)= 0.
$$
\label{conj0}\end{conjecture}

We verify this conjecture for $n\leq 8$ using Mattig's experimental data \cite{Lu}.  Theorem \ref{theointro} is also an evidence for this conjecture as the staircase partitions do satisfy the condition on $\mu_1$. 
\bigskip

The results  of this paper and \cite{L1} suggests that Conjecture \ref{conj0} could be reduced to a statement on Kronecker coefficients (see \S\ref{tensor}). However we could not find a precise reasonable conjectural statement on Kronecker coefficients  that would imply Conjecture \ref{conj0}.

\bigskip

{\bf Acknowledgment:} We are very grateful to the anonymous referee for very valuable comments which improved a lot the presentation  of the paper. We thank Ole Warnaar for useful discussions about this paper. The present work started while the first author was visiting both Sydney Mathematical Research Institute and the University of Queensland. The first author is very grateful to both institutions SMRI and the University of Queensland for their generous support. He is also grateful to Masoud Kamgarpour for the invitation at the University of Queensland and many useful discussions. The second author was supported by an Australian Government Postgraduate Award and the National Research Foundation of
Korea (NRF) grant funded by the Korea government (MSIT) (No. RS-2024-00334558).

\section{Preliminaries}

We denote by $\calP$ the set of all partitions and for a non-negative integer $n$, we let $\calP_n$ be the subset of partitions of size $n$. For a partition $\lambda=(\lambda_1,\lambda_2,\dots,\lambda_r)$ with $\lambda_1\geq\lambda_2\geq\cdots\geq\lambda_r$ we denote by $\ell(\lambda)=r$ its length and by $|\lambda|=\lambda_1+\cdots+\lambda_r$ its size.

\subsection{Roots of star-shaped graphs}\label{ss:roots}

Let us start by explaining briefly why root systems appear in our context.
\bigskip

Crawley-Boevey \cite{CB} associated to any $k$-tuple $\mathcal{O}=(\mathcal{O}_1,\dots,\mathcal{O}_k)$ of adjoint orbits of $\gl_n(K)$ (with $K$ an algebraically closed field) a star-shaped graph $\Gamma_{\mathcal{O}}$ with $k$ branches (whose lengths are determined by the Jordan type of the orbits). Using the theory of quiver representations, he found a very nice solution in terms of the root system of $\Gamma_{\mathcal{O}}$ of the Deligne-Simpson problem which is to determine for which $k$-tuples of adjoint orbits $\mathcal{O}$ the following equation has a solution

$$
X_1+\cdots+X_k=0,\hspace{1cm}X_1\in\mathcal{O}_1,\dots, X_k\in\mathcal{O}_k.
$$
In \cite{hausel-letellier-villegas2}, we use Fourier transforms to link the counting (over finite fields) of the number of solutions of the above equation with the multiplicity of the trivial character of $\GL_n(\F_q)$ in a tensor product of irreducible characters of $\GL_n(\F_q)$. We use this in \cite{L1} to study tensor products of unipotent characters of $\GL_n(\F_q)$.
\bigskip

We now recall some definitions for an arbitrary graph (see \cite[Chapter 5]{kac} for more details).

Assume given a finite graph $\Gamma=(I,\Omega)$ where $I$ is the set of vertices and $\Omega$ the set of edges. We assume that $\Gamma$ has no loops.

For $i\in I$ we let $\e_i$ be the element of $\Z^I$ defined as 

$$
(\e_i)_j=\delta_{i,j}
$$
Denote by ${\frak C}=(c_{ij})_{i,j\in I}$ the Cartan matrix of $\Gamma$, i.e.

$$
c_{ij}=\begin{cases}2 &\text{ if } i=j \\
-n_{ij}&\text{ otherwise}\end{cases}
$$
where $n_{ij}$ is the number of edges between the vertices $i$ and $j$. The Cartan matrix defines a symmetric bilinear form $(\,,\,)$ on $\Z^I$ by

$$
(\e_i,\e_j)=c_{ij}.
$$
For $i\in I$, we have the fundamental reflection $s_i:\Z^I\rightarrow\Z^I$

$$
s_i(\v)=\v-(\v,\e_i)\e_i,\hspace{1cm}\v\in\Z^I.
$$
If $n_{ij}$ is at most $1$ (which will be our case), then the $i$-th coordinate of $s_i(\v)$ equals 

$$
(\sum_j v_j)-v_i
$$
where $j$ runs over the vertices which are connected to $i$ by an edge, and the $j$-th coordinate of $s_i(\v)$, for $j\neq i$, remains unchanged (i.e. $(s_i(\v))_j=v_j$).
\bigskip

We say that $\v\in\Z^I$ is a \emph{dimension vector} if the coordinates of $\v$ are all non-negative.
\bigskip

\bigskip

\begin{example}Consider the graph
{\footnotesize{
\begin{center}
\begin{tikzpicture}[scale=.4,
    mycirc/.style={circle,fill=black!,inner sep=2pt, minimum size=.1 cm}
    ]
    \node[mycirc,label=45:{},label=45:{$4$} ] (n1) at (0,0) {};
    \node[mycirc,label=45:{$2$} ] (n2) at (5,0) {};
    % \node[mycirc,label=45:{$0$} ] (m1) at (10,0) {};
    \node[mycirc,label=45:{$2$} ] (n3) at (-5,0) {};
    % \node[mycirc,label=45:{$0$} ] (m2) at (-10,0) {};
        \node[mycirc,label=45:{$1$} ] (n4) at (0,-5) {};
         %\node[mycirc,label=45:{$0$} ] (m3) at (0,-10) {};
    \draw (n1) -- (n2) -- (n3); --(m1)  ;
\draw(n1)--(n4);--(m3);
\end{tikzpicture}
\end{center}}}
\noindent with vector dimension $\v$ whose coordinates are as indicated on the graph. Then if $s_0$ denotes the reflection at the central vertex, then the coordinate of $s_0(\v)$ at the central vertex is $2+2+1-4=1$ and the coordinates at the other vertices remain unchanged.
\label{exD4}\end{example}
\bigskip

The Weyl group $W$ of $\Gamma$ is the subgroup of automorphisms $\Z^I\rightarrow\Z^I$ generated by the fundamental reflections $\{s_i\,|\, i\in I\}$. A vector $\v\in\Z^I$ is called a \emph{real root} if $\v=w(\e_i)$ for some $i\in I$ and $w\in W$.

\bigskip

\begin{example}The dimension vector 

{\footnotesize{
\begin{center}
\begin{tikzpicture}[scale=.4,
    mycirc/.style={circle,fill=black!,inner sep=2pt, minimum size=.1 cm}
    ]
    \node[mycirc,label=45:{},label=45:{$2$} ] (n1) at (0,0) {};
    \node[mycirc,label=45:{$1$} ] (n2) at (5,0) {};
    % \node[mycirc,label=45:{$0$} ] (m1) at (10,0) {};
    \node[mycirc,label=45:{$1$} ] (n3) at (-5,0) {};
    % \node[mycirc,label=45:{$0$} ] (m2) at (-10,0) {};
        \node[mycirc,label=45:{$1$} ] (n4) at (0,-5) {};
         %\node[mycirc,label=45:{$0$} ] (m3) at (0,-10) {};
    \draw (n1) -- (n2) -- (n3); --(m1)  ;
\draw(n1)--(n4);--(m3);
\end{tikzpicture}
\end{center}}}
\noindent is a real root as it can be obtained from  $\e_0$ (where $0$ is the labeling of the central vertex) by applying first to $\e_0$ the reflections at the other vertices to get the dimension vector with coordinate $1$ everywhere and then by applying the reflection $s_0$.

\end{example}

\bigskip

The set of \emph{fundamental imaginary} roots $M$ is defined as the subset of $\v\in\Z^I\backslash\{0\}$ with connected support such that for all $i\in I$ we have

$$
(\e_i,\v)\leq 0.
$$
In the next proposition we will give a more explicit necessary and sufficient condition for a dimension vector to be in $M$ in the case of star-shaped graph.
\bigskip

Recall that an imaginary root is a vector which is of the form $w(\v)$ or $w(-\v)$ for some $\v\in M$ and $w\in W$. A root is said to be \emph{positive} if its coordinates are all non-negative. An imaginary positive root is of the form $w(\delta)$ with $\delta\in M$. The Weyl group $W$ preserves thus the set of positive imaginary roots.
\bigskip

From now we assume that $\Gamma$ is a star-shaped graph with $3$ legs as follows 
\bigskip

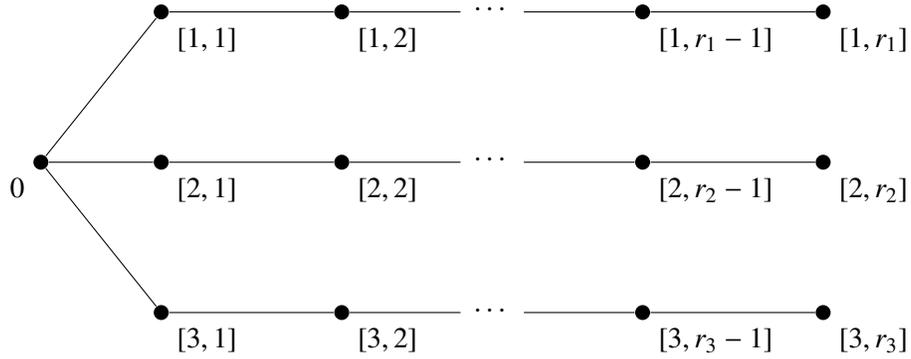
\begin{center}
\begin{tikzpicture}[scale=.4,
    mycirc/.style={circle,fill=black!,inner sep=2pt, minimum size=.1 cm}
    ]
    \node[mycirc, label=220:{$0$}] (n) at (-20,0) {};
     \node[mycirc, label=315:{$[1,1]$}] (n1) at (-16,5) {};
    \node[mycirc ,label=315:{$[1,2]$}] (n2) at (-10,5) {};
    \node[circle] (n3) at (-5,5) {$\cdots$};
    \node[mycirc ,label=315:{$[1,r_1-1]$}] (n4) at (0,5) {};
        \node[mycirc ,label=315:{$[1,r_1]$}] (n5) at (6,5) {};
        %%%%%%%%%%%%%%%5
          \node[mycirc, label=315:{$[2,1]$}] (m1) at (-16,0) {};
    \node[mycirc ,label=315:{$[2,2]$}] (m2) at (-10,0) {};
    \node[circle] (m3) at (-5,0) {$\cdots$};
    \node[mycirc ,label=315:{$[2,r_2-1]$}] (m4) at (0,0) {};
        \node[mycirc ,label=315:{$[2,r_2]$}] (m5) at (6,0) {};
         %%%%%%%%%%%%%%%5
          \node[mycirc, label=315:{$[3,1]$}] (k1) at (-16,-5) {};
    \node[mycirc ,label=315:{$[3,2]$}] (k2) at (-10,-5) {};
    \node[circle] (k3) at (-5,-5) {$\cdots$};
    \node[mycirc ,label=315:{$[3,r_3-1]$}] (k4) at (0,-5) {};
        \node[mycirc ,label=315:{$[3,r_3]$}] (k5) at (6,-5) {};
    \draw (n)--(n1) -- (n2) -- (n3) --(n4) --(n5);
     \draw (n)--(m1) -- (m2) -- (m3) --(m4) --(m5);
      \draw (n)--(k1) -- (k2) -- (k3) --(k4) --(k5);
\end{tikzpicture}
\end{center}

where $I=\{0\}\cup\{[i,j]\,|\, 1\leq i\leq 3,\, 1\leq j\leq r_i\}$.

\noindent The reflection $s_0$ at the central vertex $0$ acts on $\v=(v_i)_{i\in I}\in\Z^I$ as

$$
s_0(\v)_i=\begin{cases}v_{[1,1]}+v_{[2,1]}+v_{[3,1]}-v_0&\text{ if } i=0,\\v_i&\text{ otherwise}.\end{cases}
$$
and the other reflections $s_{[i,j]}$ act on $\v$ as

$$
s_{[i,j]}(\v)_r=\begin{cases}v_{[i,j-1]}+v_{[i,j+1]}-v_{[i,j]}&\text{ if } r=[i,j]\\v_r&\text{otherwise}.\end{cases}
$$
with $v_{[i,r_i+1]}=0$.
\bigskip

For $\v=(v_i)_i\in\Z^I$ define the integer

\begin{equation}\label{eq:delta}
\delta(\v):=v_0-\sum_{i=1}^3(v_0-v_{[i,1]}).
\end{equation}
\bigskip

For a triple $\muhat=(\mu^1,\mu^2,\mu^3)$ of partitions of $n$ of length respectively $r_1+1, r_2+1,r_3+1$, we let $\v_\muhat$ be the dimension vector of $\Gamma$ with coordinate $n$ at the central vertex $0$ and $n-\sum_{s=1}^j\mu^i_j$ at the vertex $[i,j]$. 
\bigskip

For example if $\muhat=((1^3),(1^3),(1^3))$, then the coordinates of $\v_\muhat$ are as indicated on the following graph.

{\footnotesize{
\begin{center}
\begin{tikzpicture}[scale=.4,
    mycirc/.style={circle,fill=black!,inner sep=2pt, minimum size=.1 cm}
    ]
    \node[mycirc, label=120:{$3$}] (n) at (-20,0) {};
     \node[mycirc, label=45:{$2$}] (n1) at (-15,5) {};
    \node[mycirc ,label=45:{$1$}] (n2) at (-10,5) {};
     % \node[mycirc ,label=45:{$0$}] (n3) at (-5,5) {};
%    \node[circle] (n3) at (-5,5) {$\cdots$};
        %%%%%%%%%%%%%%%5
          \node[mycirc, label=45:{$2$}] (m1) at (-15,0) {};
    \node[mycirc ,label=45:{$1$}] (m2) at (-10,0) {};
    % \node[mycirc ,label=45:{$0$}] (m3) at (-5,0) {};
%    \node[circle] (m3) at (-5,0) {$\cdots$};
         %%%%%%%%%%%%%%%5
          \node[mycirc, label=45:{$2$}] (k1) at (-15,-5) {};
    \node[mycirc ,label=45:{$1$}] (k2) at (-10,-5) {};
   %  \node[mycirc ,label=45:{$0$}] (k3) at (-5,-5) {};
%    \node[circle] (k3) at (-5,-5) {$\cdots$};
    \draw (n)--(n1) -- (n2); %--(n3)  ;
     \draw (n)--(m1) -- (m2);%--(m3)   ;
      \draw (n)--(k1) -- (k2) ;%--(k3)  ;
\end{tikzpicture}
\end{center}}}

We have the following  proposition.

\begin{proposition}\label{prop:imaginaryroot}\cite[Proposition 20]{L1}
A vector  is in $M$ if and only if it is the form $\v_\muhat$ for some triple of partitions $\muhat$ and 

\begin{equation}
\delta(\v_\muhat)=n-\mu^1_1-\mu^2_1-\mu^3_1\geq 0.
\label{eq:d}\end{equation}

\end{proposition}

Let $\N^I$ be the set of dimension vectors of $\Gamma$ and let $(\N^I)^*$ be the subset of dimension vector with non-increasing coordinates along each leg, and let $(\N^I)^{**}\subset (\N^I)^*$ be the subset of dimension vectors of the form $\v_\muhat$.

We have the following result.

\begin{proposition} (i) The subgroup $H$ of $W$ generated by the set of reflections $\{s_i\,|\, i\in I\backslash\{0\}\}$ preserves $(\N^I)^*$.

\noindent (ii) For any $\v\in(\N^I)^*$, there exists a dimension vector $\u\in(\N^I)^{**}$ such that $\v=h(\u)$ for some $h\in H$. More precisely for each $i=1,2,3$, the sequence 

$$
\sigma^i(\u):=(u_0-u_{[i,1]}, u_{[i,1]}-u_{[i,2]},\dots, u_{[i,r_i-1]}-u_{[i,r_i]}, u_{[i,r_i]})
$$
is the non-increasing sequence obtained  by re-ordering the coordinates of $\sigma^i(\v)$.
\label{ordering}\end{proposition}

\begin{proof}
    \noindent(i) It is sufficient to show that $s_{[i,j]}((\N^I)^*)=(\N^I)^*$ for any $[i,j]$. Note that $\v\in (\N^I)^*$ if and only if the coordinates of $ \sigma^i(\v)$ are non-negative for all $i$. 
Furthermore, from the definition of $s_{[i,j]}$, we see that the coordinates of $\sigma^i(s_{[i,j]}(\v))$  are obtained from those of $\sigma^i(\v)$ by permuting the $j$-th coordinate with the $(j+1)$-th coordinate, and that $\sigma^k(s_{[i,j]}(\v))=\sigma^k(\v)$ for $k\neq i$, hence the result.
    
    \noindent(ii) As the effect of the reflection $s_{[i,j]}$ on the sequence $\sigma^i(\v)$ is to permute the $j$-th coordinate with the $(j+1)$-th coordinate, we can obtain a vector dimension $\u$ such that the sequences $\sigma^i(\u)$, with $i=1,2,3$, are partitions.    
\end{proof}

For a triple  $\muhat=(\mu^1,\mu^2,\mu^3)\in\calP_n^3$, we denote by $\Gamma_\muhat$ the graph with $3$ legs as above and with $r_i:=\ell(\mu^i)-1$. Notice that $\v_\muhat$ is then a dimension vector of $\Gamma_\muhat$.

\subsection{Representations of the symmetric group and symmetric functions}\label{symmetric-section}

For a partition $\mu$ of size $n$ we denote by $\chi^\mu$ the corresponding irreducible character of the symmetric group $S_n$ and by $\chi^\mu_\lambda$ its value at an element of cycle-type $\lambda$. The trivial character is $\chi^{(n)}$ and the sign character $\chi^{(1^n)}$. 
\bigskip

We denote by $R_n$ the $\Z$-module generated by all irreducible characters of $S_n$  and we consider  the ring

$$
R=\bigoplus_{n=0}^{+\infty}R_n
$$
with $R_0=\Z$ and with product defined by 

$$
f_n\cdot f_m={\rm Ind}_{S_n\times S_m}^{S_{n+m}}(f_n\times f_m).
$$
Then $R$ is  a commutative, associative, graded ring with an identity element. It is equipped with the usual scalar product $\left\langle\,,\,\right\rangle$ making the basis of irreducible characters an orthonormal basis.
\bigskip

Let $\x=\{x_1,x_2,\dots\}$ be an infinite set of variables and let $\Lambda=\Lambda(\x)$ be the ring of symmetric functions. 

For a partition $\mu=(\mu_1,\mu_2,\dots)$ we let 

$$
p_\mu=p_\mu(\x)=(x_1^{\mu_1}+x_2^{\mu_1}+\cdots)(x_1^{\mu_2}+x_2^{\mu_2}+\cdots)\cdots
$$
be the corresponding power sum symmetric function.  It is equipped  with a scalar product (Hall pairing) such that
\bigskip

$$
\left\langle p_\lambda,p_\mu\right\rangle=\delta_{\lambda\mu}z_\lambda
$$
where $z_\lambda$ denotes the size of the centralizer in $S_{|\lambda|}$ of an element of $S_{|\lambda|}$ of cycle-type $\lambda$.

\bigskip

The \emph{Frobenius characteristic map} \cite[Chap. I, \S 7]{macdonald} is a ring isomorphism  ${\rm ch}:R\rightarrow\Lambda$ defined by

$$
{\rm ch}(f)=\sum_{|\mu|=n}z_\mu^{-1}f_\mu p_\mu,
$$
for any $f\in R_n$  where $f_\mu$ denotes the value of $f$ at the conjugacy class of cycle-type $\mu$. Moreover, ${\rm ch}$ is an isometry with respect to the scalar products.
\bigskip

Recall also that the Frobenius characteristic map of the irreducible character $\chi^\mu$  is the Schur symmetric function $s_\mu=s_\mu(\x)$.
\bigskip

Recall that the base change matrix between the two base $\{s_\mu\}_\mu$ and $\{p_\mu\}_\mu$ is given by the character table $\{\chi^\lambda_\mu\}$ of symmetric groups, more precisely

\begin{equation}
s_\lambda=\sum_{|\mu|=|\lambda|}\chi^\lambda_\mu\frac{p_\mu}{z_\mu}.
\label{Sym}\end{equation}

\bigskip
Given two partitions $\lambda$ and $\mu$ respectively of size $n$ and $m$, and a partition $\nu$ of $n+m$, the Littlewood-Richardson coefficient $c_{\lambda\mu}^\nu$ is defined by

\begin{align}
c_{\lambda\mu}^\nu:&=\left\langle \chi^\nu,{\rm Ind}_{S_n\times S_m}^{S_{n+m}}(\chi^\lambda\boxtimes\chi^\mu)\right\rangle\label{LRdef}\\
&=\left\langle s_\nu,s_\lambda s_\mu\right\rangle.\label{LR1}
\end{align}

Choose a total order on the set of all partitions and denote by ${\bf T}^o$ the set of non-increasing sequences of partitions $\omega^o=\omega^1\omega^2\cdots\omega^r$. The size of $\omega^o$ is defined as 
$$
|\omega^o|=\sum_{i=1}^r|\omega^i|
$$
and we denote by ${\bf T}^o_n$ the elements of size $n$. 

It will be convenient to write the elements of ${\bf T}^o$ in the form $(\omega^1)^{n_1}(\omega^2)^{n_2}\cdots(\omega^s)^{n_s}$ with $\omega^1>\omega^2>\cdots>\omega^s$. 
\bigskip

We will need the following generalization of the Littlewood-Richardson coefficients :
\bigskip

\noindent For $\omega^o=(\omega^1)^{n_1}(\omega^2)^{n_2}\cdots(\omega^s)^{n_s}\in{\bf T}^o_n$ and for a partition $\nu$ of $n$, we generalize (\ref{LRdef}) as 

$$
c_{\omega^o}^\nu:=\left\langle \chi^\nu,{\rm Ind}_{S_{\omega^o}}^{S_n}\left((\chi^{\omega^1})^{\boxtimes n_1}\boxtimes\cdots\boxtimes(\chi^{\omega^s})^{\boxtimes n_s}\right)\right\rangle
$$
where 

$$
S_{\omega^o}=\prod_{i=1}^s(S_{|\omega^i|})^{n_i}\subset S_n.
$$
For a partition $\mu$ we denote by $V_\mu$ an irreducible representation of $S_{|\mu|}$ affording the character $\chi^\mu$ and we put

$$
V_{\omega^o}=\bigotimes _{i=1}^s (V_{\omega^i})^{\otimes n_i}.
$$
This is an $S_{\omega^o}$-module. 

For a partition $\mu$ of $|\omega^o|$ we have

\begin{align*}
c_{\omega^o}^\mu&={\rm dim}\,{\rm Hom}_{S_n}\left(V_\mu,{\rm Ind}_{S_{\omega^o}}^{S_n}(V_{\omega_o})\right)\\
&={\rm dim}\,{\rm Hom}_{S_{\omega^o}}\left(V_\mu,V_{\omega_o}\right).
\end{align*}
The last equality is by Frobenius reciprocity.

Now the group 

$$
W_{\omega^o}:=\prod_{i=1}^s S_{n_i}
$$
acts on the space $V_{\omega^o}$ by the permutation action of $S_{n_i}$ on $(V_{\alpha^i})^{\otimes n_i}$ and can be regarded as a subgroup of the normalizer of $S_{\omega^o}$ in $S_n$ (which is the semi-direct product $S_{\omega^o}\rtimes W_{\omega^o}$). 

Therefore it acts on the space

$$
\mathbb{C}_{\omega^o}^\mu:={\rm Hom}_{S_{\omega^o}}\left(V_\mu,V_{\omega_o}\right)
$$ 
as
$$
w\cdot f(v):=w\cdot f(w^{-1}\cdot v)
$$
for $f\in\mathbb{C}_{\omega^o}^\mu$, $v\in V_{\omega^o}$ and $w\in W_{\omega^o}$.
\bigskip

We now explain how to express the term

$$
\Tr\left(w\,,\, \mathbb{C}_{\omega^o}^\mu\right)
$$
with $w\in W_{\omega^o}$ in terms of Schur functions.
\bigskip

We consider a total ordering on the set of pairs $(d,\lambda)$ where $d$ is a strictly positive integer and where $\lambda$ is a non-zero partition. Define ${\bf T}$ to be the set of \emph{types}, namely the set of sequences

$$
(d_1,\omega^1)^{m_1}(d_2,\omega^2)^{m_2}\cdots(d_s,\omega^s)^{m_s}
$$
with $(d_1,\omega^1)>(d_2,\omega^2)>\cdots>(d_s,\omega^s)$. The size of a type $\omega=\{(d_i,\omega^i)^{n_i}\}_i$ is

$$
|\omega|:=\sum_{i=1}^sn_id_i|\omega^i|
$$
and we denote by ${\bf T}_n$ the set of types of size $n$.
\bigskip

Given a family $\{f_\lambda\}_\lambda$ of symmetric functions indexed by partitions, we define for any type $\omega=\{(d_i,\omega^i)^{n_i}\}_i$, a symmetric function

$$
f_\omega:=\prod_if_{\omega^i}(\x^{d_i})^{n_i}
$$
where $\x^d$ stands for the set $\{x_1^d,x_2^d,\dots\}$.
\bigskip

Notice that giving $\omega^o=(\omega^1)^{n_1}\cdots(\omega^s)^{n_s}\in{\bf T}^o_n$ and a conjugacy class of $W_{\omega^o}$ (or equivalently a partition of $n_i$ for all $i$) defines a type $\omega\in{\bf T}_n$.
\bigskip

\begin{example} If $\omega^o=(1^2)^4(3^1)^6\in{\bf T}^o_{26}$, then $W_{\omega^o}=S_4\times S_6$ and if we choose the partitions $(2,2)$ and $(3,2,1)$, then the associated type is

$$
\omega=(2,(1^2))^2(3,(3^1))(2,(3^1))(1,(3^1))\in{\bf T}_{26}.
$$
\end{example}

We have the following proposition which generalizes Formula (\ref{LR1}).

\begin{proposition}\cite[Proposition 6.2.5]{L2} Let $\omega^o=(\omega^1)^{n_1}\cdots(\omega^s)^{n_s}\in{\bf T}^o_n$ and $w\in W_{\omega^o}$, then

$$
\Tr\left(w\,,\, \mathbb{C}_{\omega^o}^\mu\right)=\left\langle s_\mu, s_\omega\right\rangle
$$
where $\omega$ is the type defined from the pair $(\omega^o,w)$.
\label{Prop13}\end{proposition}

\begin{example}\label{ex} Consider $\omega^o=(1^2)(1)^2$ and $\mu=(3,1)$.  Then $W_{\omega^o}=S_2$ since the unique partition  $(1)$ of $1$ has multiplicity $2$.  

Using the above proposition, we wish to compute

$$
\Tr\left(\sigma, \mathbb{C}^{(3,1)}_{(1^2)(1)^2}\right)
$$
where $\sigma$ is the non-trivial element of $S_2$. 
\bigskip

Notice that the type $\omega$ associated to $(\omega^o,\sigma)$ is $\omega=(1,(1^2))(2,1)$ and so

\begin{align*}
s_\omega&=s_{(1^2)}(\x)s_{(1)}(\x^2)\\
&=s_{(1^2)}p_2(\x).
\end{align*}

Using Formula (\ref{Sym}) to express $p_2$ in terms of Schur functions together with Littlewood-Richardson coefficients, we compute 

\begin{align*}
   \left\langle s_{(3,1)}(\x),s_\omega\right\rangle&=\left\langle s_{(3,1)}(\x),s_{(1^2)}(\x)s_{(2)}(\x)-s_{(1^2)}(\x)s_{(1^2)}(\x) \right\rangle\\
    \\
    &= \left\langle s_{(3,1)}(\x),s_{(3,1)}(\x)-s_{(2^2)}(\x)-s_{(1^4)}(\x)\right\rangle=1.
\end{align*}
This shows that $\Tr\left(\sigma, \mathbb{C}^{(3,1)}_{(1^2)(1)^2}\right)=1$, equivalently,  $\mathbb{C}^{(3,1)}_{(1^2)(1)^2}$ is the trivial representation.

\end{example}

\subsection{Unipotent characters of $\GL_n(\F_q)$}\label{unipotent}

Denote by $B$ the subgroup of $\GL_n$ of upper triangular matrices and denote by $\mathbb{C}[\GL_n(\F_q)/B(\F_q)]$ the $\mathbb{C}$-vector space with basis the set $\GL_n(\F_q)/B(\F_q)$. The group $\GL_n(\F_q)$ acts by left multiplication on the latter set and so acts on $\mathbb{C}[\GL_n(\F_q)/B(\F_q)]$.
\bigskip

The \emph{unipotent representations} of $\GL_n(\F_q)$ are defined as the irreducible constituents of $\mathbb{C}[\GL_n(\F_q)/B(\F_q)]$ and they are naturally parametrized by the irreducible characters of $S_n$ and so by the partitions of $n$.

We will denote by $\calU^\mu$ the unipotent character corresponding to the partition $\mu$. Then 

$$
\calU^{(n^1)}={\rm Id}_{\GL_n(\F_q)},\hspace{1cm}\calU^{(1^n)}={\rm St}
$$
where ${\rm St}$ denotes the Steinberg character of $\GL_n(\F_q)$.

\bigskip

Given a triple $\muhat=(\mu^1,\mu^2,\mu^3)$ partitions of size $n$, we consider the multiplicities

$$
g_\muhat=\left\langle \chi^{\mu^1}\otimes\chi^{\mu^2}\otimes\chi^{\mu^3},1\right\rangle_{S_n},\hspace{1cm} U_\muhat(q):=\left\langle \calU^{\mu^1}\otimes\calU^{\mu^2}\otimes\calU^{\mu^3},1\right\rangle_{\GL_n(\F_q)}.
$$
Recall that the first one is a non-negative integer  known as a \emph{Kronecker coefficient} while the second one is a polynomial in $q$ with non-negative integer coefficients (see \cite[Theorem 4]{L1}).

\begin{theorem}\cite[Proposition 6]{L1} If $g_\muhat\neq 0$ then $U_\muhat(q)\neq 0$. In fact the term $g_\muhat$ contributes to the constant term of $U_\muhat(q)$ (i.e. $g_\muhat\leq U_\muhat(0)$).
\label{L1-1}\end{theorem}

The converse of the above theorem is false. For instance if $\muhat=((1^n),(1^n),(1^n))$, then $g_\muhat=0$ but $U_\muhat(q)\neq 0$ (see \cite[\S 3.6]{L1} for $n=3$).
\bigskip

\subsection{Generalities on tensor products of unipotent characters}

We consider the set $\overline{\calP}$ of triples of partitions of the same size and we repeat what we did in \S \ref{symmetric-section} with $\overline{\calP}$ instead of $\calP$. 

Namely we choose  a total ordering on $\overline{\calP}$ and we denote by $\overline{\bold T}_n^o$ the set of sequences $\omegahat^o=(\omegahat^1)^{n_1}(\omegahat^2)^{n_2}\cdots(\omegahat^s)^{n_s}$ of elements of $\overline{\calP}$ with $\omegahat^1>\cdots>\omegahat^s$ such that $\sum_in_i|\omega^i|=n$.

Define 

$$
{\bold S}_{\omegahat^o}:=\prod_{i=1}^s({\bold S}_{|\omegahat^i|})^{n_i}\subset {\bold S}_n,\hspace{1cm}{\bold V}_{\omega^o}:=\bigotimes_{i=1}^s({\bold V}_{\omegahat^i})^{\otimes n_i},\hspace{1cm}W_{\omegahat^o}:=\prod_{i=1}^s S_{n_i}
$$
where for $n\in\N^*$ and $\muhat=(\mu^1,\mu^2,\mu^3)\in\overline{\calP}$ we put

$$
{\bold S}_n=S_n\times S_n\times S_n,\hspace{1cm}{\bold V}_\muhat=V_{\mu^1}\boxtimes V_{\mu^2}\boxtimes V_{\mu^3}.
$$
Define the $W_{\omegahat^o}$-module

$$
\C^\muhat_{\omegahat^o}:={\rm Hom}_{{\bold S}_{\omegahat^o}}\left({\bold V}_{\omegahat^o},{\bold V}_\muhat\right).
$$
\begin{remark}Notice that we have a natural map $\overline{{\bold T}}_n^o\rightarrow ({\bold T}_n^o)^3, \omegahat^o\mapsto (\omega^o_1,\omega^o_2,\omega^o_3)$ where $\omega^o_i$ is obtained by mapping each $\omegahat^j\in\overline{\calP}$ in $\omegahat^o=(\omegahat^1)^{n_1}\cdots(\omegahat^s)^{n_s}$ to its $i$-th coordinate. 

Notice that the multiplicities of partitions in $\omega_i^o$ may be larger than $n_1,n_2,\dots,n_s$ : 
 For instance if we take $\omegahat^o=((1^2),(2^1),(2^1))((2^1),(1^2),(2^1))$, then $\omega_1^o=\omega_2^o=(1^2)(2^1)$ and $\omega^o_3=(2^1)^2$. 

As vector spaces, we have

$$
\C^\muhat_{\omegahat^o}=\bigotimes_{i=1}^3\mathbb{C}^{\mu^i}_{\omega_i ^o}
$$
where $\muhat=(\mu^1,\mu^2,\mu^3)$. Via the diagonal embedding of $W_{\omegahat^o}$ in $W_{\omega^o_1}\times W_{\omega^o_2}\times W_{\omega^o_3}$, this is an isomorphism of $W_{\omegahat^o}$-modules.

\end{remark}
\bigskip

We have the following result.

\begin{theorem} \cite[Corollary 5]{L1} Let $\muhat=(\mu^1,\mu^2,\mu^3)\in\overline{\calP}$ be of size $n$ (i.e. $|\mu^1|=|\mu^2|=|\mu^3|=n$). Assume that there exists $\omegahat^o=(\omegahat^1)^{n_1}\cdots(\omegahat^s)^{n_s}\in\overline{\bold T}_n^o$ such that the two following conditions are satisfied:

\noindent (1)  The dimension vector $\v_{\omegahat^i}$ is a root of $\Gamma_{\omegahat^i}$ for all $i=1,\dots,s$, 

\noindent (2) $\left\langle \C^\muhat_{\omegahat^o},1\right\rangle_{W_{\omegahat^o}}\neq 0$.
\bigskip

\noindent Then $U_\muhat(q)\neq 0$.
\label{L1-3}\end{theorem}

\begin{lemma}Theorem \ref{L1-3} implies Theorem \ref{L1-1}.
\end{lemma}

\begin{proof} If we take $\omegahat^o=((1),(1),(1))^n$, then the condition (1) of the above theorem  is satisfied as the graph is the Dynkin diagram $A_1$ with dimension vector $1$ at the unique vertex. 

Moreover ${\bold S}_{\omegahat^o}=({\bold S}_1)^n$ is the trivial subgroup of ${\bold S}_n$, $W_{\omegahat^o}=S_n$ and ${\bold V}_{\omegahat^o}=\mathbb{C}$. Therefore for $\muhat=(\mu^1,\mu^2,\mu^3)$ of size $n$,

$$
\C^\muhat_{\omegahat^o}={\rm Hom}(\mathbb{C},V_{\mu^1}\boxtimes V_{\mu^2}\boxtimes V_{\mu^3})=V_{\mu^1}\boxtimes V_{\mu^2}\boxtimes V_{\mu^3}
$$
as $S_n$-modules (for the diagonal action). Therefore $\left\langle \C^\muhat_{\omegahat^o},1\right\rangle_{W_{\omegahat^o}}=g_\muhat$ and so if it is non-zero then so is $U_\muhat(q)$ by Theorem \ref{L1-3}.  
\end{proof}
\bigskip

\begin{remark}Theorem \ref{L1-3} also implies that if $\v_\muhat$ is a root then $U_\muhat(q)\neq 0$ as in this case we can take $\muhat$ for $\omegahat^o$.
\label{L1-2}\end{remark}
\bigskip

\begin{example}Let us give an example where the condition of Theorem \ref{L1-3} is satisfied but $\v_\muhat$ is not a root and $g_\muhat= 0$. Moreover this example will be used later.

Consider 
$$
\muhat=((2^2),(2^2),(3,1)).
$$
Then $(\Gamma_\muhat,\v_\muhat)$ is as follows

{\footnotesize{
\begin{center}
\begin{tikzpicture}[scale=.4,
    mycirc/.style={circle,fill=black!,inner sep=2pt, minimum size=.1 cm}
    ]
    \node[mycirc,label=45:{},label=45:{$4$} ] (n1) at (0,0) {};
    \node[mycirc,label=45:{$2$} ] (n2) at (5,0) {};
    % \node[mycirc,label=45:{$0$} ] (m1) at (10,0) {};
    \node[mycirc,label=45:{$2$} ] (n3) at (-5,0) {};
    % \node[mycirc,label=45:{$0$} ] (m2) at (-10,0) {};
        \node[mycirc,label=45:{$1$} ] (n4) at (0,-5) {};
         %\node[mycirc,label=45:{$0$} ] (m3) at (0,-10) {};
    \draw (n1) -- (n2) -- (n3); --(m1)  ;
\draw(n1)--(n4);--(m3);
\end{tikzpicture}
\end{center}}}

We can check that the corresponding Kronecker coefficient $g_\muhat$ vanishes (and so the condition of Theorem \ref{L1-1} is not satisfied) and  that $\v_\muhat$ is not a root (by applying first the central reflection to $\v_\muhat$ and then  reflections at other vertices). Let us see that the condition of Theorem \ref{L1-3} is satisfied.

Take $\omegahat^o=\omegahat^1(\omegahat^2)^2$ with $\omegahat^1=((1^2),(1^2),(1^2))$ and $\omegahat^2=((1),(1),(1))$. Then $\v_{\alpha^1}$ is the longest root of the Dynkin diagram $D_4$ and $\omegahat^2$ is the unique positive root of $A_1$. We have
$$
\mathbb{C}_{(1^2)(1)^2}^{(2^2)}=\mathbb{C},\hspace{1cm}\mathbb{C}_{(1^2)(1)^2}^{(3,1)}=\mathbb{C}.
$$
We need to compute the action of $W_{\omegahat^o}=S_2$ on these two spaces. Denote by $\sigma$ the non-trivial element of $S_2$, by Example \ref{ex} we have
$$
{\rm Tr}\left(\sigma\,,\, \mathbb{C}^{(3,1)}_{(1^2)(1)^2}\right)=1.
$$ 
 Therefore the action of $S_2$  on $\mathbb{C}_{(1^2)(1)^2}^{(3,1)}$ is trivial and so  $S_2$ acts trivially on

$$
\C_{\omegahat^o}^\muhat=\mathbb{C}_{(1^2)(1)^2}^{(2^2)}\otimes\mathbb{C}_{(1^2)(1)^2}^{(2^2)}\otimes\mathbb{C}_{(1^2)(1)^2}^{(3,1)}
$$
as the tensor square of $\mathbb{C}_{(1^2)(1)^2}^{(2^2)}$ must be the trivial $S_2$-module (in fact a calculation shows that $\mathbb{C}_{(1^2)(1)^2}^{(2^2)}$ is the sign representation of $S_2$). We thus deduce that
$$
\left\langle \C^\muhat_{\omegahat^o},1\right\rangle_{W_{\omegahat^o}}=1.
$$
The condition of Theorem \ref{L1-3} is thus satisfied and so we have  $U_\muhat(q)\neq 0$. The result of \cite{L1} is more precise and tells us that

$$
U_\muhat(q)=\left\langle \C^\muhat_{\omegahat^o},1\right\rangle_{W_{\omegahat^o}}=1
$$
which is consistent with the experimental data.
\label{example}\end{example}

\section{An analogue of the Saxl conjecture for unipotent characters}\label{s:mainthm}

Fix a positive integer $d$, put $n=\sum_{i=1}^d i$ and consider the partition  $\xi_d=(d,d-1, \ldots , 1)$ of $n$. 

Recall that the Saxl conjecture says that for any partition $\tau$ of $n$ we have

$$
g_{(\xi_d,\xi_d,\tau)}\neq 0.
$$
The Saxl conjecture implies the following theorem which is the main theorem of this paper.

\begin{theorem}\label{thm:analoguesaxl}
    For any partition $\tau$ of $n$ we have  
    $$
    U_{(\xi_d,\xi_d,\tau)}(q)\neq  0.
    $$
\label{maintheo}\end{theorem}
Let us choose a partition $\tau=(\tau_1,\tau_2, \ldots, \tau_s)$ of $n$. The proof of this  theorem is divided into three cases : 

(1) $\tau_1\leq n-d$ and $d\geq 7$,

(2) $\tau_1 \leq n-d$ and $d<6$,

(3) $\tau_1>n-d$.
\bigskip

To ease the notation we will denote the multipartition $(\xi_d,\xi_d,\tau)$ by $\tauhat$.

\subsection{The case $\tau_1\leq n-d$ and $d \geq 7$}

The aim of this section is to prove the following result.

\begin{proposition} If $\tau_1\leq n-d$  and $d\geq 7$ then the dimension vector $\v_\tauhat$ is an imaginary root of $\Gamma_\tauhat$.
\label{theo1}\end{proposition}

Together with Remark \ref{L1-2}, this implies that Theorem \ref{maintheo} is true when $\tau_1\leq n-d$ and $d\geq 7$.
\bigskip

\begin{lemma} If $\tau_1\leq n-2d$, then $\v_\tauhat$ is a fundamental imaginary root of $\Gamma_\tauhat$.
\end{lemma}

\begin{proof}From Equation \eqref{eq:d}, we have $\delta(\v_\tauhat)=n-(2d+\tau_1)$. 

If $\tau_1\leq n-2d$, then 
$$
\delta(\v_\tauhat)=n-(2d+\tau_1)\geq 0
$$
from which we deduce that $\v_\tauhat$ is a fundamental imaginary root by Proposition \ref{prop:imaginaryroot}. 
\end{proof}
\bigskip

To prove Proposition \ref{theo1} we are thus reduced to prove the following result:

\begin{proposition}If $n-2d<\tau_1\leq n-d$  and $d\geq 7$ then the dimension vector $\v_\tauhat$ is an imaginary root of $\Gamma_\tauhat$.
\label{theo1'}\end{proposition}

\begin{proof}Let us first explain the strategy of the proof.

\noindent We assume that
$$
n-2d< \tau_1\leq n-d
$$
and we write $\tau_1=n-2d+k$ for some integer $1\leq k\leq d$. 
\bigskip

Notice that 

$$
\delta(\v_\tauhat)=n-2d-\tau_1=-k
$$
and so $\v_\tauhat$ is not a fundamental imaginary root. To see that $\v_\tauhat$ is an imaginary root, we will construct an element $w\in W$ such that 

$$
\delta(w(\v_\tauhat))\geq 0,
$$
 and $w(\v_\tauhat)$ has the form $\v_\muhat$ i.e. such that $w(\v_\tauhat)$ is a fundamental imaginary root. As $W$ preserves positive imaginary roots, this will prove that $\v_\tauhat$ is a positive  imaginary root.

\bigskip

To construct such a $w$ we will use the following process. We apply the central reflection $s_0$ to $\v_\muhat$. This will have the good effect of increasing the $\delta$ value but will have the wrong effect of creating a dimension vector which is not in the form $\v_\muhat$, for some $\muhat$, anymore. Therefore we will apply reflections at other vertices to get a dimension vector of the form $\v_\muhat$ using Proposition \ref{ordering}(ii). At the end of the first step of the process we will have a dimension vector $\v_\tauhat^{(1)}$ of the form $\v_\muhat$ such that $\delta(\v_\tauhat^{(1)})>\delta(\v_\tauhat)$. If $\delta(\v_\tauhat^{(1)})<0$, we repeat the process with $\v_\tauhat^{(1)}$ and we keep going  until we get a dimension vector $\v_\tauhat^{(m)}$ with a non-negative $\delta$.
\bigskip

We now explain this in detail. We first assume that $d\geq 8$. Notice that $(\Gamma_\tauhat,\v_\tauhat)$ is as follows

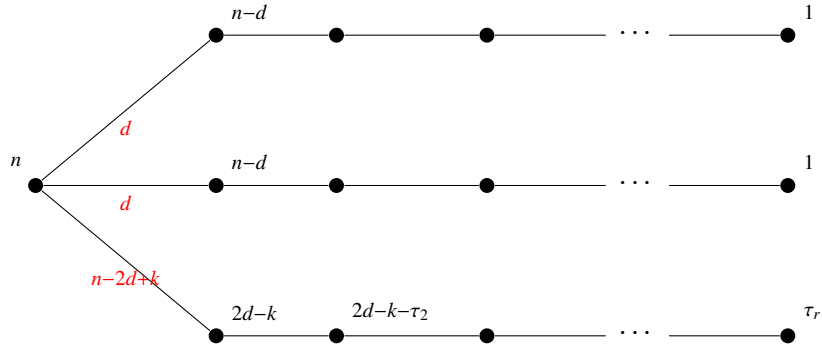
\begin{center}
\begin{tikzpicture}[scale=.4,
    mycirc/.style={circle,fill=black!,inner sep=2pt, minimum size=.1 cm}
    ]
    \node[mycirc, label=120:{$\substack{n}$}] (n) at (-25,0) {};
     \node[mycirc, label=45:{$\substack{n-d}$}] (n1) at (-19,5) {};
    \node[mycirc, label=45:{$\substack{n-2d+1}$}](n2) at (-15,5) {};
     \node[mycirc, label=45:{$\substack{n-3d+3}$}] (n6) at (-10,5) {};
    \node[circle] (n3) at (-5,5) {$\cdots$};
    \node[mycirc ,label=45:{$\substack{1}$ }] (n4) at (0,5) {};
       % \node[mycirc ,label=45:{$\substack{0}$}] (n5) at (6,5) {};
        %%%%%%%%%%%%%%%5
          \node[mycirc, label=45:{$\substack{n-d}$}] (m1) at (-19,0) {};
    \node[mycirc, label=45:{$\substack{n-2d+1}$}] (m2) at (-15,0) {};
        \node[mycirc, label=45:{$\substack{n-3d+3}$}] (m6) at (-10,0) {};
    \node[circle] (m3) at (-5,0) {$\cdots$};
    \node[mycirc ,label=45:{$\substack{1}$ }] (m4) at (0,0) {};
       % \node[mycirc ,label=45:{$\substack{0}$}] (m5) at (6,0) {};
         %%%%%%%%%%%%%%%5
          \node[mycirc, label=45:{$\substack{2d-k}$}] (k1) at (-19,-5) {};
    \node[mycirc ,label=45:{$\substack{2d-k-\tau_2} $}] (k2) at (-15,-5) {};
        \node[mycirc, label=45:{$\substack{2d-k-\tau_2-\tau_3}$}] (k6) at (-10,-5) {};
    \node[circle] (k3) at (-5,-5) {$\cdots$};
    \node[mycirc ,label=45:{$\substack{\tau_r}$}] (k4) at (0,-5) {};
       % \node[mycirc ,label=45:{$\substack{0}$}] (k5) at (6,-5) {};
    \draw  (n1) -- node[below]{$\substack{\textcolor{red}{d-1}}$} (n2) -- node[below]{$\substack{\textcolor{red}{d-2}}$}  (n6)--node[below]{$\substack{\textcolor{red}{d-3}}$}  (n3) --node[below]{$\substack{\textcolor{red}{2}}$} (n4);% --(n5);
     \draw  (m1) --node[below]{$\substack{\textcolor{red}{d-1}}$}  (m2) -- node[below]{$\substack{\textcolor{red}{d-2}}$} (m6)--node[below]{$\substack{\textcolor{red}{d-3}}$} (m3) --node[below]{$\substack{\textcolor{red}{2}}$} (m4);% --(m5);
      \draw  (k1) -- node[below]{$\substack{\textcolor{red}{\tau_2}}$} (k2) --node[below]{$\substack{\textcolor{red}{\tau_3}}$}(k6)-- node[below]{$\substack{\textcolor{red}{\tau_4}}$}(k3) --node[below]{$\substack{\textcolor{red}{\tau_{r-1}}}$}(k4);% --(k5);
      \draw (n) -- node[below]{$\substack{\textcolor{red}{d}}$} (n1);
        \draw (n) -- node[below] {$\substack{\textcolor{red}{d}}$} (m1);
      \draw (n) -- node[below] {$\substack{\textcolor{red}{n-2d+k}}$} (k1);
\end{tikzpicture}
\end{center}
where the black color integers  indicate the coordinates of $\v_\tauhat$ and the red color integers on the edges denote the successive differences between the coordinates of $\v_\tauhat$, i.e. $\sigma^i(\v_\tauhat)$.
\bigskip

We now apply the central reflection $s_0$ and we get $(\Gamma_\tauhat,s_0(\v_\tauhat))$

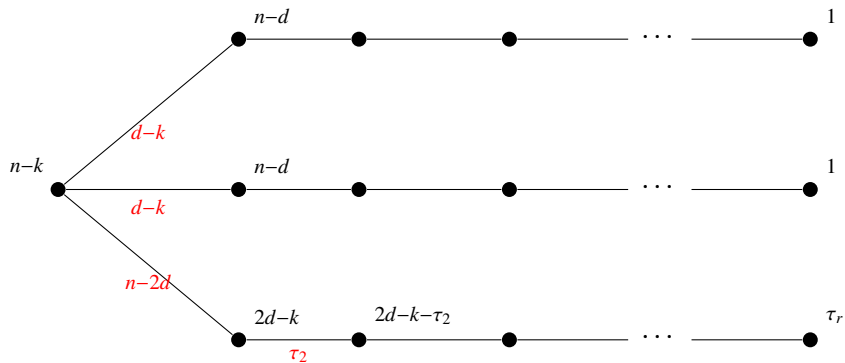
\begin{center}
\begin{tikzpicture}[scale=.4,
    mycirc/.style={circle,fill=black!,inner sep=2pt, minimum size=.1 cm}
    ]
    \node[mycirc, label=120:{$\substack{n-k}$}] (n) at (-25,0) {};
     \node[mycirc, label=45:{$\substack{n-d}$}] (n1) at (-19,5) {};
    \node[mycirc, label=45:{$\substack{n-2d+1}$}](n2) at (-15,5) {};
     \node[mycirc, label=45:{$\substack{n-3d+3}$}] (n6) at (-10,5) {};
    \node[circle] (n3) at (-5,5) {$\cdots$};
    \node[mycirc ,label=45:{$\substack{1}$ }] (n4) at (0,5) {};
       % \node[mycirc ,label=45:{$\substack{0}$}] (n5) at (6,5) {};
        %%%%%%%%%%%%%%%5
          \node[mycirc, label=45:{$\substack{n-d}$}] (m1) at (-19,0) {};
    \node[mycirc, label=45:{$\substack{n-2d+1}$}] (m2) at (-15,0) {};
        \node[mycirc, label=45:{$\substack{n-3d+3}$}] (m6) at (-10,0) {};
    \node[circle] (m3) at (-5,0) {$\cdots$};
    \node[mycirc ,label=45:{$\substack{1}$ }] (m4) at (0,0) {};
       % \node[mycirc ,label=45:{$\substack{0}$}] (m5) at (6,0) {};
         %%%%%%%%%%%%%%%5
          \node[mycirc, label=45:{$\substack{2d-k}$}] (k1) at (-19,-5) {};
    \node[mycirc ,label=45:{$\substack{2d-k-\tau_2} $}] (k2) at (-15,-5) {};
        \node[mycirc, label=45:{$\substack{2d-k-\tau_2-\tau_3}$}] (k6) at (-10,-5) {};
    \node[circle] (k3) at (-5,-5) {$\cdots$};
    \node[mycirc ,label=45:{$\substack{\tau_r}$}] (k4) at (0,-5) {};
       % \node[mycirc ,label=45:{$\substack{0}$}] (k5) at (6,-5) {};
    \draw  (n1) -- node[below]{$\substack{\textcolor{red}{d-1}}$} (n2) -- node[below]{$\substack{\textcolor{red}{d-2}}$}  (n6)--node[below]{$\substack{\textcolor{red}{d-3}}$}  (n3) --node[below]{$\substack{\textcolor{red}{2}}$} (n4);% --(n5);
     \draw  (m1) --node[below]{$\substack{\textcolor{red}{d-1}}$}  (m2) -- node[below]{$\substack{\textcolor{red}{d-2}}$} (m6)--node[below]{$\substack{\textcolor{red}{d-3}}$} (m3) --node[below]{$\substack{\textcolor{red}{2}}$} (m4);% --(m5);
      \draw  (k1) -- node[below]{$\substack{\textcolor{red}{\tau_2}}$} (k2) --node[below]{$\substack{\textcolor{red}{\tau_3}}$}(k6)-- node[below]{$\substack{\textcolor{red}{\tau_4}}$}(k3) --node[below]{$\substack{\textcolor{red}{\tau_{r-1}}}$}(k4);% --(k5);
      \draw (n) -- node[below]{$\substack{\textcolor{red}{d-k}}$} (n1);
        \draw (n) -- node[below] {$\substack{\textcolor{red}{d-k}}$} (m1);
      \draw (n) -- node[below] {$\substack{\textcolor{red}{n-2d+k}}$} (k1);
\end{tikzpicture}
\end{center}
We have $s_0(\v_\tauhat)\in (\N^I)^*$ and  $\delta(s_0(\v_\tauhat))=0$ but $s_0(\v_\tauhat)$ is not anymore in $(\N^I)^{**}$ as we now see in detail.

Since $\tau$ is a partition of $n$ we have $\tau_1+\tau_2\leq n$ and so  $2d-k\geq\tau_2$. Moreover as $d\geq 8$, we have

\begin{equation}
n=\frac{d(d+1)}{2}\geq 4d
\label{eq1}\end{equation}
from which we get $n-2d\geq 2d-k$. Therefore the successive differences between the coordinates of the dimension vector on the last leg define a partition as $n-2d\geq 2d-k\geq\tau_2$, i.e. $$\sigma^3(s_0(\v_\tauhat))=(n-2d+k,\tau_2, \tau_3, \ldots, \tau_r)$$ is a partition. This is not the case for the first two legs if $k>1$, i.e. $$\sigma^1(s_0(\v_\tauhat))=\sigma^2(s_0(\v_\tauhat))=(d-k,d-1,d-2,\ldots , 2,1)$$ is not a partition if $k>1$. But  we may apply some element in $H$ to obtain a dimension vector $\v^{(1)}_\tauhat$ in $(\N^I)^{**}$, i.e. move the first $d-k$ to come after $d-k+1$, so that the sequence becomes $(d-1,d-2, \ldots , d-k+1,d-k,d-k,d-k-1, \ldots,1)$ (see proof of Proposition \ref{ordering}(ii)).
More explicitly, we apply $s_{[2,k-1]}\cdots s_{[2,1]}s_{[1,k-1]}\cdots s_{[1,1]}$ to $s_0(\v_{\tauhat})$ to  get $(\Gamma_\tau,\v^{(1)}_\tauhat)$ with $\v^{(1)}_\tauhat\in (\N^I)^{**}$:

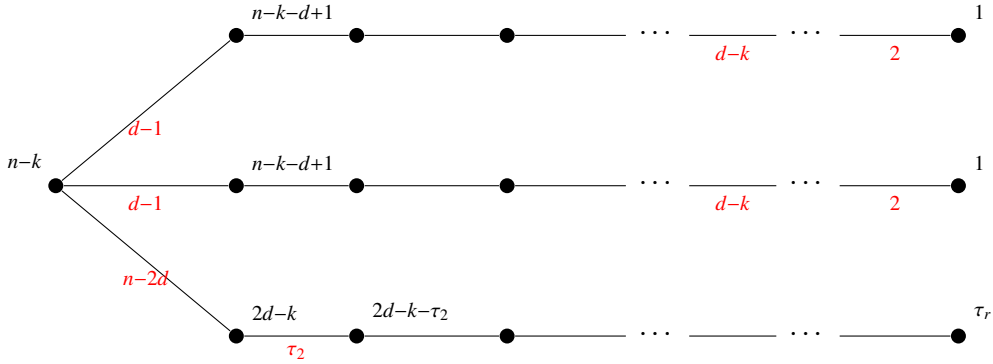
\begin{center}
\begin{tikzpicture}[scale=.4,
    mycirc/.style={circle,fill=black!,inner sep=2pt, minimum size=.1 cm}
    ]
    \node[mycirc, label=120:{$\substack{n-k}$}] (n) at (-30,0) {};
     \node[mycirc, label=45:{$\substack{n-k-d+1}$}] (n1) at (-24,5) {};
    \node[mycirc, label=45:{$\substack{n-k-2d+3}$}] (n2) at (-20,5) {};
     \node[mycirc, label=45:{$\substack{n-k-3d+6}$}] (n6) at (-15,5) {};
     \node[circle] (11) at (-10,5) {$\cdots$};
    \node[circle] (n3) at (-5,5) {$\cdots$};
    \node[mycirc ,label=45:{$\substack{1}$ }] (n4) at (0,5) {};
       % \node[mycirc ,label=45:{$\substack{0}$}] (n5) at (6,5) {};
        %%%%%%%%%%%%%%%5
          \node[mycirc, label=45:{$\substack{n-k-d+1}$}] (m1) at (-24,0) {};
    \node[mycirc, label=45:{$\substack{n-k-2d+3}$}] (m2) at (-20,0) {};
        \node[mycirc, label=45:{$\substack{n-k-3d+6}$}] (m6) at (-15,0) {};
         \node[circle] (22) at (-10,0) {$\cdots$};
    \node[circle] (m3) at (-5,0) {$\cdots$};
    \node[mycirc ,label=45:{$\substack{1}$}] (m4) at (0,0) {};
     %   \node[mycirc ,label=45:{$\substack{0}$}] (m5) at (6,0) {};
         %%%%%%%%%%%%%%%5
          \node[mycirc, label=45:{$\substack{2d-k}$}] (k1) at (-24,-5) {};
    \node[mycirc ,label=45:{$\substack{2d-k-\tau_2 }$}] (k2) at (-20,-5) {};
        \node[mycirc ,label=45:{$\substack{2d-k-\tau_2-\tau_3 }$}] (k6) at (-15,-5) {};
            \node[circle] (33) at (-10,-5) {$\cdots$};
    \node[circle] (k3) at (-5,-5) {$\cdots$};
    \node[mycirc ,label=45:{$\substack{\tau_r}$ }] (k4) at (0,-5) {};
        %\node[mycirc ,label=45:{$\substack{0}$}] (k5) at (6,-5) {};
        \draw (n3)--node[below]{$\substack{\textcolor{red}{d-k}}$} (11);
    \draw  (n1) -- node[below]{$\substack{\textcolor{red}{d-2}}$}(n2) --node[below]{$\substack{\textcolor{red}{d-3}}$}(n6)--node[below]{$\substack{\textcolor{red}{d-4}}$}(11);
    \draw (n3) --node[below]{$\substack{\textcolor{red}{2}}$}(n4);% --node[below]{$\substack{\textcolor{red}{1}}$}(n5);
     \draw  (m1) -- node[below]{$\substack{\textcolor{red}{d-2}}$}(m2) -- node[below]{$\substack{\textcolor{red}{d-3}}$}(m6)--node[below]{$\substack{\textcolor{red}{d-4}}$}(22);
     \draw(m3) --node[below]{$\substack{\textcolor{red}{2}}$}(m4);% --node[below]{$\substack{\textcolor{red}{1}}$}(m5);
     \draw (22)--node[below]{$\substack{\textcolor{red}{d-k}}$} (m3);
      \draw   (k2)--node[below]{$\substack{\textcolor{red}{\tau_3}}$}(k6)-- node[below]{$\substack{\textcolor{red}{\tau_4}}$}(33);
      \draw(k3) --node[below]{$\substack{\textcolor{red}{\tau_{r-1}}}$}(k4);% --(k5);
      \draw(33)--node[below]{} (k3);
      \draw (n) -- node[below] {$\substack{\textcolor{red}{d-1}}$} (n1);
        \draw (n) -- node[below] {$\substack{\textcolor{red}{d-1}}$} (m1);
      \draw (n) -- node[below] {$\substack{\textcolor{red}{n-2d}}$} (k1);
      \draw (k1) -- node[below]{$\substack{\textcolor{red}{\tau_2}}$} (k2);
\end{tikzpicture}
\end{center}

As shown  just before the graph, in the first two legs there are two consecutive edges labelled by $d-k$ and we have

$$
\delta(\v^{(1)}_\tauhat)=-k+2.
$$
Therefore if $k=1$ or $2$ then $\v^{(1)}_\tauhat$ is a fundamental imaginary root and we are done. If $k\geq 3$ we keep going with the process, i.e. we apply $s_0$ first and then use the reflections of $H$ to re-order the labels of the edges  in a non-decreasing way from right to left since $s_0(\v^{(1)}_\tauhat)\in (\N^I)^*$. Then we get a dimension vector $\v^{(2)}_\tau\in(\N^I)^{**}$ and $(\Gamma_\tauhat,\v^{(2)}_\tau )$ as follows

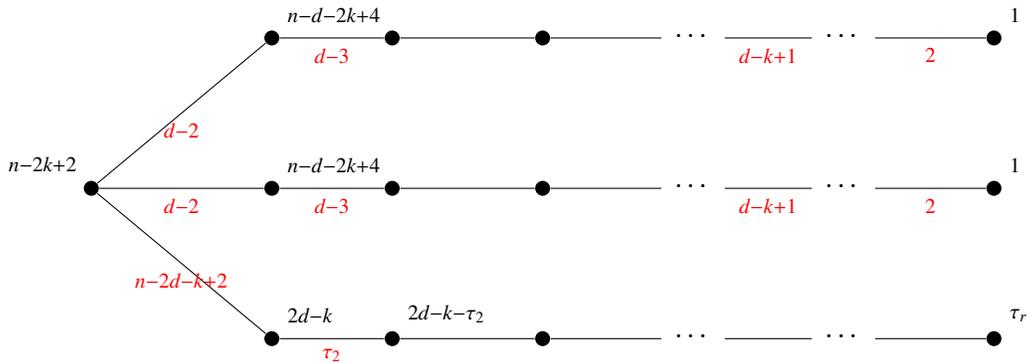
\begin{center}
\begin{tikzpicture}[scale=.4,
    mycirc/.style={circle,fill=black!,inner sep=2pt, minimum size=.1 cm}
    ]
    \node[mycirc, label=120:{$\substack{n-2k+2}$}] (n) at (-30,0) {};
     \node[mycirc, label=45:{$\substack{n-d-2k+4}$}] (n1) at (-24,5) {};
    \node[mycirc ,label=45:{}][mycirc, label=45:{$\substack{n-2d-2k+7}$}] (n2) at (-20,5) {};
     \node[mycirc, label=45:{$\substack{n-3d-2k+11}$}] (n6) at (-15,5) {};
     \node[circle] (11) at (-10,5) {$\cdots$};
    \node[circle] (n3) at (-5,5) {$\cdots$};
    \node[mycirc ,label=45:{$\substack{1}$ }] (n4) at (0,5) {};
       % \node[mycirc ,label=45:{$\substack{0}$}] (n5) at (6,5) {};
        %%%%%%%%%%%%%%%5
          \node[mycirc, label=45:{$\substack{n-d-2k+4}$}] (m1) at (-24,0) {};
    \node[mycirc, label=45:{$\substack{n-2d-2k+7}$}] (m2) at (-20,0) {};
        \node[mycirc, label=45:{$\substack{n-3d-2k+11}$}] (m6) at (-15,0) {};
         \node[circle] (22) at (-10,0) {$\cdots$};
    \node[circle] (m3) at (-5,0) {$\cdots$};
    \node[mycirc ,label=45:{$\substack{1}$}] (m4) at (0,0) {};
       % \node[mycirc ,label=45:{$\substack{0}$}] (m5) at (6,0) {};
         %%%%%%%%%%%%%%%5
          \node[mycirc, label=45:{$\substack{2d-k}$}] (k1) at (-24,-5) {};
    \node[mycirc ,label=45:{$\substack{2d-k-\tau_2 }$}] (k2) at (-20,-5) {};
        \node[mycirc ,label=45:{$\substack{2d-k-\tau_2 -\tau_3}$}] (k6) at (-15,-5) {};
            \node[circle] (33) at (-10,-5) {$\cdots$};
    \node[circle] (k3) at (-5,-5) {$\cdots$};
    \node[mycirc ,label=45:{$\substack{\tau_r}$ }] (k4) at (0,-5) {};
      %  \node[mycirc ,label=45:{$\substack{0}$}] (k5) at (6,-5) {};
        \draw (n3)--node[below]{$\substack{\textcolor{red}{d-k+1}}$} (11);
    \draw  (n1) --node[below] {$\substack{\textcolor{red}{d-3}}$}  (n2) --node[below] {$\substack{\textcolor{red}{d-4}}$} (n6)--node[below] {$\substack{\textcolor{red}{d-5}}$} (11);
    \draw (n3) --node[below]{$\substack{\textcolor{red}{2}}$}(n4);% --node[below]{$\substack{\textcolor{red}{1}}$}(n5);
     \draw  (m1) --node[below] {$\substack{\textcolor{red}{d-3}}$}  (m2) --node[below] {$\substack{\textcolor{red}{d-4}}$}  (m6)--node[below] {$\substack{\textcolor{red}{d-5}}$} (22);
     \draw(m3) --node[below]{$\substack{\textcolor{red}{2}}$}(m4);% --node[below]{$\substack{\textcolor{red}{1}}$}(m5);
     \draw (22)--node[below]{$\substack{\textcolor{red}{d-k+1}}$} (m3);
      \draw   (k2)--node[below] {$\substack{\textcolor{red}{\tau_3}}$} (k6)-- node[below] {$\substack{\textcolor{red}{\tau_4}}$} (33);
      \draw(k3) --node[below] {$\substack{\textcolor{red}{\tau_{r-1}}}$} (k4);% --(k5);
      \draw(33)--node[below]{ } (k3);
      \draw (n) -- node[below] {$\substack{\textcolor{red}{d-2}}$} (n1);
        \draw (n) -- node[below] {$\substack{\textcolor{red}{d-2}}$} (m1);
      \draw (n) -- node[below] {$\substack{\textcolor{red}{n-2d-k+2}}$} (k1);
      \draw (k1) -- node[below]{$\substack{\textcolor{red}{\tau_2}}$} (k2);
\end{tikzpicture}
\end{center}

Notice that from (\ref{eq1}) we have

$$
n-2d-k+2\geq 2d-k
$$
and so $n-2d-k+2\geq \tau_2$ and so the successive differences of the coordinates on the last leg form a partition, i.e. $\sigma^3(\v^{(1)}_\tauhat)=(n-2d-k+2,\tau_2,\tau_3,\ldots , \tau_r)$ is a partition.

We have

$$
\delta(\v^{(2)}_\tauhat)=-k+4.
$$
So if $k\leq 4$ then $\v^{(2)}_\tauhat$ is a fundamental imaginary root and the process stops. If $k>4$ we start again the process and for $0\leq m\leq \lceil k/2\rceil$ we end up, after $m$ iterations of our labelling algorithm,  with a dimension vector $\v_\tauhat^{(m)}\in(\N^I)^{**}$ which has the form
\bigskip
 
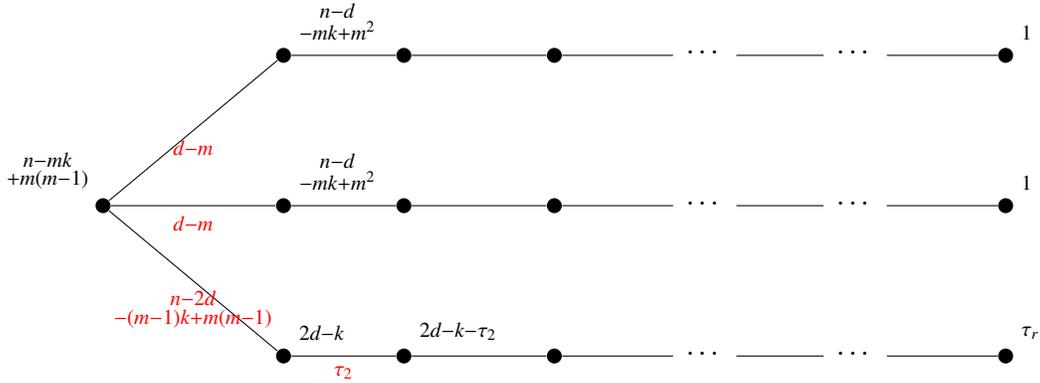
\begin{center}
\begin{tikzpicture}[scale=.4,
    mycirc/.style={circle,fill=black!,inner sep=2pt, minimum size=.1 cm}
    ]
    \node[mycirc, label=120:{$\substack{n-mk\\+m(m-1)}$}] (n) at (-30,0) {};
     \node[mycirc, label=45:{$\substack{n-d \\ -mk+m^2}$}] (n1) at (-24,5) {};
    \node[mycirc ,label=45:{}] (n2) at (-20,5) {};
     \node[mycirc ,label=45:{}] (n6) at (-15,5) {};
     \node[circle] (11) at (-10,5) {$\cdots$};
    \node[circle] (n3) at (-5,5) {$\cdots$};
    \node[mycirc ,label=45:{$\substack{1}$ }] (n4) at (0,5) {};
       % \node[mycirc ,label=45:{$\substack{0}$}] (n5) at (6,5) {};
        %%%%%%%%%%%%%%%5
          \node[mycirc, label=45:{$\substack{n-d\\ -mk+m^2}$}] (m1) at (-24,0) {};
    \node[mycirc ,label=45:{ }] (m2) at (-20,0) {};
        \node[mycirc ,label=45:{ }] (m6) at (-15,0) {};
         \node[circle] (22) at (-10,0) {$\cdots$};
    \node[circle] (m3) at (-5,0) {$\cdots$};
    \node[mycirc ,label=45:{$\substack{1}$}] (m4) at (0,0) {};
       % \node[mycirc ,label=45:{$\substack{0}$}] (m5) at (6,0) {};
         %%%%%%%%%%%%%%%5
          \node[mycirc, label=45:{$\substack{2d-k}$}] (k1) at (-24,-5) {};
    \node[mycirc ,label=45:{$\substack{2d-k-\tau_2 }$}] (k2) at (-20,-5) {};
        \node[mycirc ,label=45:{ }] (k6) at (-15,-5) {};
            \node[circle] (33) at (-10,-5) {$\cdots$};
    \node[circle] (k3) at (-5,-5) {$\cdots$};
    \node[mycirc ,label=45:{$\substack{\tau_r}$ }] (k4) at (0,-5) {};
       % \node[mycirc ,label=45:{$\substack{0}$}] (k5) at (6,-5) {};
        \draw (n3)--node[below]{$\substack{\textcolor{red}{d-k+m-1}}$  }  (11);
    \draw  (n1) --node[below] { }  (n2) --(n6)--(11);
    \draw (n3) --node[below]{ }(n4);% --node[below]{ }(n5);
     \draw  (m1) --node[below] { }  (m2) -- (m6)--  (22);
     \draw(m3) --node[below]{$\substack{\textcolor{red}{2}}$  }(m4) ;%--node[below]{ }(m5);
     \draw (22)--node[below]{$\substack{\textcolor{red}{d-k+m-1}}$  } (m3);
      \draw   (k2)--node[below]{$\substack{\textcolor{red}{\tau_3}}$}   (k6)--node[below]{$\substack{\textcolor{red}{\tau_4}}$}  (33);
      \draw(k3) --node[below]{$\substack{\textcolor{red}{\tau_{r-1}}}$} (k4);% --(k5);
      \draw(33)--(k3);
      \draw (n) -- node[below] {$\substack{\textcolor{red}{d-m}}$} (n1);
        \draw (n) -- node[below] {$\substack{\textcolor{red}{d-m}}$} (m1);
      \draw (n) -- node[below] {$\substack{\textcolor{red}{n-2d}\\ \textcolor{red}{-(m-1)k+m(m-1)}}$} (k1);
      \draw (k1) -- node[below]{$\substack{\textcolor{red}{\tau_2}}$} (k2);
\end{tikzpicture}
\end{center}
Note that we can check that $s_0(\v^{(j)}_\tauhat)\in (\N^I)^*$ for all $j =0,1,2,\ldots , \lceil k/2\rceil$ inductively.

The key point in the induction process is that we never need to use reflections in $H$ on the last leg to re-organise the labels on the edges in a non-increasing way, i.e. when $m\leq\lceil k/2\rceil$ we always have

$$
n-2d-(m-1)k+m(m-1)\geq 2d-k\geq \tau_2,
$$
equivalently, $\sigma^3(\v^{(m)}_\tauhat)=(n-2d-(m-1)k+m(m-1),\tau_2, \ldots , \tau_r)$ is a partition.
Indeed, as the sequence $(n-mk+m(m-1))_{m\leq\lceil k/2\rceil}$ is decreasing, we only need to check the above inequality for $m=\lceil k/2\rceil$, i.e. we need to check that

\begin{equation}
2d^2-14d\geq\begin{cases}k^2-6k &\text{if } k\text{ is even,}\\k^2-6k+1&\text{ if } k\text{ is odd.}\end{cases}
\label{ineq2}\end{equation}
But these inequalities are true under our assumption $0\leq k\leq d$ and $d\geq 8$.

We have
$$
\delta(\v_\tauhat^{(m)})=-k+2m
$$
and so for $m=\lceil k/2\rceil$ we get that $\v^{(m)}_\tauhat$ is a fundamental imaginary root which proves the theorem for $d\geq 8$.
\bigskip

The first part of the proof works as the inequality (\ref{eq1}) still holds for $d=7$. However the end of the proof needs to be modified as the inequalities (\ref{ineq2}) are not verified anymore for $d=7$ and $k=7$. We can check this case by hand as follows: $\delta(\v_\tauhat^{(5)})=1$ for $\tau=(21,7)$, $\delta(\v_\tauhat^{(4)})=0$ for $\tau=(21,6,1)$ and $\delta(\v_\tauhat^{(4)})=1$ for all $\tau=(21,\tau_2,\ldots , \tau_l)$ such that $\tau_2 \leq 5$.
\end{proof}

\subsection{Remaining case $\tau_1\leq n-d$ and $d \leq 6$}
It remains to verify the theorem for the cases where $d\leq 6$ and $\tau_1 \leq n-d$.
In this range, there exist elements $\v_\tauhat$ 
which are not imaginary roots; rather, they are real roots or correspond to a non-zero Kronecker coefficient. Recall that if $\v_\tauhat$  is a root (respectively, if $g_\tauhat\neq 0$), then by Remark \ref{L1-2} (respectively, by Theorem \ref{L1-1}), $U_\tauhat(q)\neq 0$.

The findings are summarized in the following table.

 \begin{figure}[H]
\centering
\begin{tabular}{|c|c|c|c|c|}
\hline
 $d$& $n$ & $\tau=(\tau_1,\tau_2,\dots)$ & $\v_\tauhat$ &$U_{\tauhat}(q)\neq 0$ because \\
\hline\hline
 $6$&$21$ & \makecell{$\substack{    \tau_1 \leq  15\ \text{except}\ (15,6)}$} &
Imaginary root & $\v_\tauhat$ is a root \\
\hline
 $6$&$21$ & \makecell{$\substack{ (15,6)}$} &
$\substack{\text{Under reflections,}\\ \text{equivalent to the real root} \\ ((1^3),(1^3),(2,1)) }$ & $\v_\tauhat$ is a root \\
\hline\hline
  $5$&$15$ & \makecell{$\substack{   \tau_1 \leq  10 \\ \text{except}\ (9,6),(10,4,1),(10,5)}$} &
Imaginary root &$\v_\tauhat$ is a root \\
\hline
 $5$&$15$ & \makecell{$\substack{  (10,4,1)}$} &
$\substack{\text{Under reflections,}\\ \text{equivalent to the real root} \\ ((1^2),(1^2),(1^2)) }$ &$\v_\tauhat$ is a root \\
\hline
  $5$&$15$ & \makecell{$\substack{  (10,5)}$} &
Not a root& $g_\tauhat=141$ \\
\hline
 $5$&$15$ & \makecell{$\substack{ (9,6)}$} &
$\substack{\text{Under reflections,}\\ \text{equivalent to the real root} \\ ((1^3),(1^3),(2,1)) }$&$\v_\tauhat$ is a root \\
\hline\hline
$4$&$10$ & \makecell{$\substack{\tau_1\leq 5\  \text{except}\ (5,4,1),(5^2) }$} &
Imaginary root &$\v_\tauhat$ is a root  \\
\hline
$4$&$10$ & \makecell{$\substack{(5,4,1)}$} &
$\substack{\text{Under reflections,}\\ \text{equivalent to the real root} \\ ((1^2),(1^2),(1^2)) }$ &$\v_\tauhat$ is a root \\
\hline
$4$&$10$ & \makecell{$\substack{(5^2)}$} &
Not a root &$g_\tauhat=6$\\
\hline
$4$&$10$ & \makecell{$\substack{(6,1^4)}$} &
Imaginary root &$\v_\tauhat$ is a root \\
\hline
$4$&$10$ & \makecell{$\substack{(6,2,1^2)}$} &
$\substack{\text{Under reflections,}\\ \text{equivalent to the real root} \\ ((1^2),(1^2),(1^2)) }$ &$\v_\tauhat$ is a root \\
\hline
$4$&$10$ & \makecell{$\substack{  (6,2^2) }$} &
Not a root & $g_\tauhat=39$\\
\hline
$4$&$10$ & \makecell{$\substack{ (6,3,1) }$} &
Not a root & $g_\tauhat=54$\\
\hline
$4$&$10$ & \makecell{$\substack{(6,4)}$} &
Not a root & $g_\tauhat=15$\\
\hline\hline
$3$&$6$ & \makecell{$\substack{\tau_1=1,2\ \text{except}\ (2^2,1^2), (2^3)}$} &
Imaginary root &$\v_\tauhat$ is a root \\
\hline
$3$&$6$ & \makecell{$\substack{(2^2,1^2)}$} &
$\substack{\text{Under reflections,}\\ \text{equivalent to the real root} \\ ((1^2),(1^2),(1^2)) }$ &$\v_\tauhat$ is a root \\
\hline
$3$&$6$ & \makecell{$\substack{(2^3)}$} &
Not a root &$g_\tauhat=2$ \\
\hline
$3$&$6$ & \makecell{$\substack{(3,1^3)}$} &
Not a root &$g_\tauhat=4$\\
\hline
$3$&$6$ & \makecell{$\substack{(3,2,1)}$} &
Not a root &$g_\tauhat=5$\\
\hline
$3$&$6$ & \makecell{$\substack{(3^2)}$} &
Not a root &$g_\tauhat=2$\\
\hline\hline
$2$&$3$ & \makecell{$\substack{(1^3)}$} &
Not a root &$g_\tauhat=1$\\
\hline\hline
$1$&$1$ & \makecell{$\substack{(1)}$} &
Real root &$\v_\tauhat$ is a root \\
\hline
\end{tabular}
\end{figure}

\subsection{The case $\tau_1>n-d$}

Recall that the dominance order  $\trianglelefteq$ is defined as follows: for two partitions $\lambda$ and $\mu$ of size $n$, $\lambda \trianglelefteq \mu$ if $\lambda_1+\cdots +\lambda_i \leq \mu_1+\cdots +\mu_i$ for all $i$.

In \cite{Ik}, it is proved that if $\xi_d$ is comparable with $\tau$, namely $\tau \trianglelefteq \xi_d$ or $\xi_d \trianglelefteq \tau$, then the Kronecker coefficient $\delta_\tauhat \neq 0$ and so by Theorem \ref{L1-1}, $U_\tauhat(q)\neq 0$.

Hence in the case $\tau_1>n-d$,  to prove Theorem \ref{maintheo}  it is sufficient to prove the following result.

\begin{lemma}If $\tau_1>n-d$, then $\xi_d\trianglelefteq \tau$.
\end{lemma}

\begin{proof}We need to prove that

$$
\tau_1+\cdots+\tau_i\geq d+(d-1)+\cdots+(d-i+1)
$$
for all $i$.
By assumption on $\tau_1$ note that

$$
\tau_1+\cdots+\tau_i\geq (n-d+1)+(i-1)=\frac{d(d-1)}{2}+i.
$$
Therefore it is enough to prove that

$$
\frac{d(d-1)}{2}+i\geq \frac{i(2d-i+1)}{2}
$$
which is equivalent to 

$$
(d-i)^2\geq d-i.
$$
Note that because of the assumption $\tau_1>n-d$, the length of $\tau$ is at most $d$ and so $d-i\geq 0$ from which we deduce the above inequality.
\end{proof}

\section{The tensor square of unipotent characters}\label{tensor}

In this section we are interested  in the unipotent characters of $\GL_n(\F_q)$ whose tensor square contains as a direct summand all the unipotent characters of $\GL_n(\F_q)$. We saw in the previous section that this is the case of the unipotent characters attached to the staircase partitions. We also know that the Steinberg characters satisfy this property \cite[Proposition 33]{L1} (see also \cite{HSTZ} for other groups). In this section we give a conjectural necessary and sufficient condition for a unipotent character to verify this property.
\bigskip

\subsection{Main result and conjecture}

The aim of this section is to bring some evidences for the following conjecture.

\begin{conjecture} Let $\mu=(\mu_1,\mu_2,\dots)$ be a partition of $n$, then 

(i) $U_{(\mu,\mu,\tau)}(q)\neq 0$ for all partitions $\tau$ of $n$  if and only if $\mu_1\leq \lceil n/2\rceil$.

(ii) If  $\mu_1>\lceil n/2\rceil$, then 

$$
U_{(\mu,\mu,(1^n))}(q)= 0.
$$
\label{conj}\end{conjecture}

We can check that the above conjecture is correct for $n\leq 8$ from Mattig's experimental data \cite{Lu}.
\bigskip

Notice that if $\mu'\trianglelefteq\mu$ and if $\mu_1\leq \lceil n/2\rceil$ then $\mu'_1\leq \lceil n/2\rceil$ and so the first assertion of Conjecture \ref{conj} implies the conjecture below.

\begin{conjecture}If for some partition $\mu$ of $n$ we have 

$$
U_{(\mu,\mu,\tau)}(q)\neq 0,\hspace{.5cm}\text{  for all partitions }\tau,
$$
then for all partitions $\mu'\trianglelefteq \mu$, we have 
$$
U_{(\mu',\mu',\tau)}(q)\neq 0,\hspace{.5cm}\text{  for all partitions }\tau.
$$

\label{weak}\end{conjecture}
\bigskip

Notice that the set $\{\mu\,|\, \mu_1\leq \lceil n/2\rceil\}$ has a maximal element $\mu^{\rm max}$ with respect to the dominance order. If $n=2k$ is even, $\mu^{\rm max}$ is the partition with two parts $(k,k)$ and if $n=2k+1$ is odd, it is the partition $(k+1,k)$.

\begin{theorem}For any partition $\mu\trianglelefteq\mu^{\rm max}$ we have 
$$
U_{(\mu,\mu,(1^n))}(q)\neq 0.
$$
\label{theo2}\end{theorem}

\begin{proof} Put 

$$
\omegahat^o:=\begin{cases}((1^2),(1^2),(1^2))^k&\text{ if } n=2k,\\
((1^2),(1^2),(1^2))^k((1),(1),(1))&\text{ if }n=2k+1.\end{cases}
$$
By Theorem \ref{L1-3} it is enough to prove that

\begin{equation}
\left\langle \C_{\omegahat^o}^{(\mu,\mu,(1^n))},1\right\rangle\neq 0.
\label{ine}\end{equation}
Put

$$
\omega^o:=\begin{cases}(1^2)^k&\text{ if } n=2k,\\
(1^2)^k(1)&\text{ if }n=2k+1\end{cases}
$$
so we have

$$
\C_{\omegahat^o}^{(\mu,\mu,(1^n))}=\mathbb{C}_{\omega^o}^\mu\otimes\mathbb{C}_{\omega^o}^\mu\otimes\mathbb{C}_{\omega^o}^{(1^n)}.
$$
To prove (\ref{ine}) it is enough to prove the following two statements :

(1) For all partitions $\mu$ of $n$ such that $\mu\trianglelefteq \mu^{\rm max}$, we have $\mathbb{C}_{\omega^o}^\mu\neq 0$,

(2) $\mathbb{C}_{\omega^o}^{(1^n)}$ is the trivial $S_k$-module.
\bigskip

By Proposition \ref{Prop13} we have

$$
c_{\omega^o}^\mu:={\rm dim}\, \mathbb{C}_{\omega^o}^\mu=\left\langle s_\mu,s_{\omega^o}\right\rangle.
$$
Notice that

$$
s_{\omega^o}:=\begin{cases}(s_{(1^2)})^k&\text{ if }n=2k,\\
(s_{(1^2)})^k s_{(1)}&\text{if }n=2k+1.\end{cases}
$$
We focus on proving (1) and (2) for even \(n = 2k\); the case of odd \(n\) will be addressed briefly, as the method is essentially the same.
\bigskip

By Pieri's formula, for two partitions $\lambda$ and $\mu$ of size $r$ we have
$$
\mathbb{C}^\mu_{\lambda (1^r)}\neq 0
$$
if $\mu$ can be obtained from $\lambda$ by adding $r$ boxes with no two in the same row. In this case it is of dimension $1$. 

\bigskip

We now prove (1) by induction on $k$. The case $k=1, 2$ are easy to verify. We thus assume that (1) is true for $k<m$.

Let us prove it for $k=m$. We have

$$
s_{\omega^o}=s_{(1^2)}^m=s_{(1^2)}^{m-1}\cdot s_{(1^2)}=\left(\sum_{\tau}c_{(1^2)^{m-1}}^\tau s_\tau\right)\cdot s_{(1^2)},
$$
where by induction hypothesis the coefficient $c_{(1^2)^{m-1}}^\tau$ is $> 0$ when $\tau\trianglelefteq (m-1,m-1)$.
\bigskip

Let $\mu$ be a partition of $2m$ such that $\mu\trianglelefteq (m,m)$. Then $\mu$ is one of the following kind :

(i) $\mu=(m,m)$,

(ii)  $\mu=(m,\mu_2,\dots,\mu_\ell)$ with $\mu_2 <m$,

(iii) $\mu=(\mu_1,\mu_2,\cdots,\mu_{\ell-1},\mu_\ell)$ with $\mu_1<m$ (in which case $\ell\geq 3$).

\bigskip

Any partition $\mu$ as above can be obtained from a partition smaller than $(m-1,m-1)$ by adding two boxes no two in the same row: in cases (i) and (iii), $\mu$ is obtained  from $(\mu_1,\mu_2,\dots,\mu_{\ell-1}-1,\mu_\ell-1)$, and in case (ii) it is obtained from  $(m-1,\dots,\mu_\ell-1)$. Hence it follows from Pieri's formula and the induction hypothesis that $s_\mu$ appears non-trivially in $s_{\omega^o}$ .

\bigskip

When \(n\) is odd, for any partition \(\mu \trianglelefteq (m, m)\), \(\mu\) must be of one of the following forms: 
\[
\mu = (m+1, m), \quad \mu = (m+1, \mu_2, \dots, \mu_\ell) \text{ with } \mu_2 < m, \quad \text{or} \quad \mu = (\mu_1, \mu_2, \dots, \mu_{\ell-1}, \mu_\ell) \text{ with } \mu_1 < m,
\]
in which case \(\ell \geq 3\). 
As in the even case, each such \(\mu\) can be obtained from a partition smaller or equal to \((m, m)\) by adding one box: specifically, from \((m, m)\), \((m, \mu_2, \dots, \mu_\ell)\), or \((\mu_1, \mu_2, \dots, \mu_{\ell-1}, \mu_\ell - 1)\). 
Hence, by Pieri’s formula and the induction hypothesis, it follows that \(s_\mu\) appears non-trivially in \(s_{\omega^o}\).

\bigskip

Let us now prove (2).

Using Pieri's formula it is clear that $s_{(1^n)}$ appears in $s_{\omegahat^o}$ with multiplicity $1$ and so 
$\mathbb{C}_{\omega^o}^{(1^n)}$ is one-dimensional.
\bigskip

Let us prove that $\mathbb{C}_{\omega^o}^{(1^n)}$ is the trivial module. As it is one-dimensional, it can be either the trivial module or the sign module. 
 To determine whether it is trivial or sign, it is enough to compute the trace of an odd permutation.

Let us assume that $k$ is even, then $k$-cycle  $\sigma=(1,2,3,\dots,k)\in S_k$ is an odd permutation.
From Proposition \ref{Prop13}, we have

\begin{align*}
\Tr\left(\sigma\,,\, \mathbb{C}_{\omega^o}^{(1^n)}\right)&=\left\langle s_{(1^n)}({\bf x}),s_{(1^2)}({\bf x}^k)\right\rangle\\
&=\left\langle s_{(1^n)}({\bf x}),\frac{1}{2}\left(p_{(1^2)}({\bf x}^k)-p_{(2^1)}({\bf x}^k)\right)\right\rangle\\
&=\left\langle s_{(1^n)}({\bf x}),\frac{1}{2}\left(p_{(k^2)}({\bf x})-p_{(n)}({\bf x})\right)\right\rangle\\
&=\frac{1}{2}\left(\langle s_{(1^n)}(\x),p_{(k^2)}(\x)\rangle-\langle s_{(1^n)}(\x),p_{(n)}(\x)\rangle\right).
\end{align*}
Since for two partitions $\lambda$ and $\mu$ we have

$$
\langle s_\lambda(\x),p_\mu(\x)\rangle=\chi^\lambda_\mu,
$$
we deduce that 
$$
\Tr\left(\sigma\,,\, \mathbb{C}_{\omega^o}^{(1^n)}\right)=1
$$
and so we proved that $\mathbb{C}^{(1^n)}_{\omega^o}$ is the trivial module when $n=2k$ with $k$ even.

If $n=2k$ with $k$ odd, we take the cycle $\sigma=(1,2,\dots,k-1)$. Analogously to the case when $k$ is even, we have 
\begin{align*}
\Tr\left(\sigma\,,\, \mathbb{C}_{\omega^o}^{(1^n)}\right)
&=\left\langle s_{(1^n)}({\bf x}),s_{(1^2)}({\bf x}^{k-1})s_{(1^2)}({\bf x})\right\rangle
\\
&=  \left\langle s_{(1^n)}(\x),\frac{1}{4}\left( p_{(k-1^2,1^2)}(\x) - p_{(2k-2,1^2)}(\x) -p_{(k-1^2,2)}(\x) +  p_{(2k-2, 2)}(\x)\right)\right\rangle  \\
&=1
\end{align*}
 and we proved that $\mathbb{C}^{(1^n)}_{\omega^o}$ is also trivial in this case.
 
 \bigskip
 
 When \(n = 2k + 1\), we again aim to show that
 \[
 \Tr\left(\sigma,\, \mathbb{C}_{\omega^o}^{(1^n)}\right) = \frac{1}{2} \left( \langle s_{(1^n)}(\mathbf{x}), p_{(k^2,1)}(\mathbf{x}) \rangle - \langle s_{(1^n)}(\mathbf{x}), p_{(2k,1)}(\mathbf{x}) \rangle \right) = 1
 \]
 when \(k\) is even, and
 \[
 \Tr\left(\sigma,\, \mathbb{C}_{\omega^o}^{(1^n)}\right) = \left\langle s_{(1^n)}(\mathbf{x}), \frac{1}{4} \left( p_{(k-1^2,1^2,1)}(\mathbf{x}) - p_{(2k-2,1^2,1)}(\mathbf{x}) - p_{(k-1^2,2,1)}(\mathbf{x}) + p_{(2k-2,2,1)}(\mathbf{x}) \right) \right\rangle = 1
 \]
 when \(k\) is odd.
 
 Note that a partition \((p_1, \ldots, p_\ell)\) is even (respectively, odd) if and only if the partition \((p_1, \ldots, p_\ell, 1)\) is even (respectively, odd). This observation implies that
 \[
 \langle s_{(1^n)}(\mathbf{x}), p_{(\mu_1, \ldots, \mu_\ell)}(\mathbf{x}) \rangle = \langle s_{(1^n)}(\mathbf{x}), p_{(\mu_1, \ldots, \mu_\ell, 1)}(\mathbf{x}) \rangle.
 \]
 This shows that the representation \(\mathbb{C}_{\omega^o}^{(1^n)}\) is again trivial.
\end{proof}

\subsection{Examples}

From Conjecture \ref{weak}, the assertion 

$$
U_{(\mu^{\rm max},\mu^{\rm max},\tau)}\neq 0,\hspace{.5cm}\text{ for all }\tau,
$$
 is thus an essential case which explains why we will focus on $\mu^{\rm max}$ in the next examples.
\bigskip

\begin{remark}\label{rem:maxi}Notice that in the case of the smallest partition $\mu=(1^n)$, the dimension vector corresponding to $(\mu,\mu,\tau)$ is  a root for any partition $\tau$  (see \cite[Proof of Proposition 33]{L1}) from which we get that  $U_{(\mu,\mu,\tau)}(q)\neq 0$ for all $\tau$. We already saw that for larger partitions $\mu$ (like the staircase partition) the dimension vector $(\mu,\mu,\tau)$ may not be always a root and so one needs to mix with other arguments like Ikenmeyer's result on Kronecker coefficients (\cite[Theorem 2.1]{Ik}). We thus observe at a computational level that  the proof of the assertion (for $\mu\trianglelefteq\mu^{\rm max}$)

$$
U_{(\mu,\mu,\tau)}\neq 0,\hspace{.5cm}\text{ for all }\tau,
$$
gets more complicated as $\mu$ gets larger corroborating the idea that $\mu=\mu^{\rm max}$ would be the essential case.
\end{remark}
\bigskip

We put

$$
M_{\omegahat^o}^\tau:=\left\langle \C_{\omegahat^o}^{(\mu,\mu,\tau)},1\right\rangle_{W_{\omegahat^o}}.
$$
Recall (see below Theorem \ref{L1-3}) that $M_{((1),(1),(1))^n}^\tau$ is the Kronecker coefficient $g_{(\mu,\mu,\tau)}$.
\bigskip

{\bf Example 1 : The case $n=4$ and $\mu=(2,2)$}
\bigskip

We have the following table 

\begin{figure}[h]
\centering
\begin{tabular}{|c|c|c|c|c|c|c|}
\hline
$\tau$ &$(4)$ 
& $(3,1)$ 
&$(2^2)$
 &$(2,1^2)$
 &$(1^4)$   \\
 \hline
 $M_{((1),(1),(1))^4}^{\tau}=g_{(\mu,\mu,\tau)}$ & $1$ &$0$&$1$&$0$&$1$\\ 
\hline
$M_{((1^2),(1^2),(1^2))^2}^\tau$ &$0$&$0$&$1$&$0$&$1$\\
\hline
$M_{((1^2),(1^2),(1^2))((1),(1),(1))^2}^\tau$ & $0$ &$1$&$0$&$1$&$0$\\ 
\hline
$M_{((1^3),(1^3),(1^3))((1),(1),(1))}^\tau$ & $0$ &$0$&$0$&$0$&$0$\\ 
\hline
$M_{((1^3),(1^3),(2,1))((1),(1),(1))}^\tau$ & $0$ &$0$&$0$&$0$&$0$\\ 
\hline
\hline
$U_{(\mu,\mu,\tau)}(q)$&$1$&$1$&$2$&$1$&$2$\\
\hline
\end{tabular}
\end{figure}

\begin{remark}In this example the dimension vector associated to $(\mu,\mu,\tau)$, when $\tau$ runs over the partitions of $4$, is never a root. Moreover the Kronecker coefficient does vanish (second and fourth columns) and so we can not  use only Theorem \ref{L1-1} and Remark \ref{L1-2} to prove the non-vanishing of $U_{(\mu,\mu,\tau)}(q)$ as  we did in the case of staircase partitions.
\end{remark}

All the roots involved non-trivially (i.e. the dimension vectors corresponding to the multipartitions $((1^2),(1^2),(1^2))$ and $((1),(1),(1))$) being real we get $U_{(\mu,\mu,\tau)}(q)$  by summing the first five rows (see  \cite[Formula (17), Proposition 27]{L1}). Notice that the sixth row  is consistent with Mattig's experimental data \cite{Lu}.

We already proved that (see Example \ref{example})

$$
M_{((1^2),(1^2),(1^2))((1),(1),(1))^2}^{(3,1)}=1.
$$
The computation of the other entries of the above table is similar.
\bigskip

{\bf Example 2 : The case $n=9$ and $\mu=(5,4)$}
\bigskip

The computation of $U_{(\mu,\mu,\tau)}(q)$ in \cite{Lu} are available for $n\leq 8$.  In the example we push the computation to $n=9$ to verify that $U_{(\mu,\mu,\tau)}(q)\neq 0$ as predicted by our conjecture. 

This example illustrates also the use of Theorem \ref{L1-3} for larger values of $n$ : The following table shows that for any partition $\tau$ of $9$, we can find an $\omegahat^o$ such that $M_{\omegahat^o}^\tau\neq 0$ from which we deduce, thanks to Theorem \ref{L1-3}, that  $U_{(\mu,\mu,\tau)}(q)\neq 0$ for all partitions $\tau$ of $9$ (as predicted by our conjecture).

{\footnotesize{
\begin{figure}[H]
\centering
\begin{tabular}{|c|c|c|c|c|}
\hline
$\tau$ &$M_{((1),(1),(1))^9}^\tau$ 
& $M_{((1^2),(1^2),(1^2))((1),(1),(1))^7}^\tau$ 
&$M_{((1^2),(1^2),(1^2))^3,((1),(1),(1))^3}^\tau$
 &$M_{((1^2),(1^2),(1^2))^4((1),(1),(1)) }^\tau$\\ \hline
 $(9)$ & $1$ &$0$&$0$&$0$\\ \hline
$(8,1)$ & $1$ &$1$&$0$&$0$\\ \hline
$(7,2)$ & $1$ &$1$&$0$&$0$\\ \hline
$(7,1^2)$ & $1$ &$2$&$0$&$0$\\ \hline
$(6,3)$ & $1$ &$1$&$1$&$0$\\ \hline
$(6,2,1)$ & $1$ &$3$&$0$&$0$\\ \hline
$(6,1^3)$ & $1$ &$2$&$0$&$0$\\ \hline
$(5,4)$ & $1$ &$1$&$0$&$1$\\ \hline
$(5,3,1)$ & $1$ &$3$&$1$&$0$\\ \hline
$(5,2^2)$ & $1$ &$2$&$0$&$0$\\ \hline
$(5,2,1^2)$ & $1$ &$4$&$1$&$0$\\ \hline
$(5,1^4)$ & $0$ &$2$&$0$&$0$\\ \hline
$(4^2,1)$ & $1$ &$2$&$0$&$1$\\ \hline
$(4,3,2)$ & $1$ &$3$&$1$&$0$\\ \hline
$(4,3,1^2)$ & $1$ &$4$&$0$&$1$\\ \hline
$(4,2^2,1)$ & $1$ &$4$&$1$&$0$\\ \hline
$(4,2,1^3)$ & $0$ & $3$&$1$&$0$\\ \hline
$(4,1^5)$ & $0$ &$1$&$1$&$0$\\ \hline
$(3^3)$ & $1$ &$1$&$1$&$0$\\ \hline
$(3^2,2,1)$ & $1$ &$4$&$0$&$1$\\ \hline
$(3^2,1^3)$ & $0$ &$3$&$0$&$1$\\ \hline
$(3,2^3)$ & $1$ &$2$&$0$&$1$\\ \hline
$(3,2^2,1^2)$ & $0$ &$3$&$1$  &$0$\\ \hline
$(3,2,1^4)$ & $0$ &$1$&$0$&$1$\\ \hline
$(3,1^6)$ & $0$ &$0$&$1$&$0$\\ \hline
$(2^4,1)$ & $0$ &$1$&$0$&$1$\\ \hline
$(2^3,1^3)$ & $0$ &$1$&$0$&$1$\\ \hline
$(2^2,1^5)$ & $0$ &$0$&$0$&$1$\\ \hline
$(2,1^7)$ & $0$ &$0$&$0$ &$1$\\ \hline
$(1^9)$ & $0$ &$0$&$0$ &$1$ \\ \hline
\end{tabular}
\end{figure}}}
\bigskip

To establish the above table we need the decomposition of  
$$
\C_{((1^2),(1^2),(1^2))((1),(1),(1))^7}^{(\mu,\mu,\tau)},\hspace{.5cm}\C_{((1^2),(1^2),(1^2))^3((1),(1),(1))^3}^{(\mu,\mu,\tau)},\hspace{.2cm}\text{ and }\hspace{.2cm}\C_{((1^2),(1^2),(1^2))^4((1),(1),(1))}^{(\mu,\mu,\tau)}$$into irreducible $S$-modules (where $S$ is respectively $S_7$, $S_3\times S_3$, and $S_4$). We need thus the decomposition of the modules   $\mathbb{C}_{(1^2)(1)^7}^\tau$, $\mathbb{C}_{(1^2)^3(1)^3}^\tau$, $\mathbb{C}_{(1^2)^4(1)}^\tau$, for any partition $\tau$,  into irreducible modules. This is given by the following table ($\cdot$ means that $\mathbb{C}_{(*)}^\tau$ has zero dimension). 

\begin{figure}[H]
\centering
\begin{tabular}{|c|c|c|c|}
\hline
\makecell{$\tau$  }
& \makecell{$\mathbb{C}_{(1^2)(1)^7}^{\tau} $} 
& \makecell{$\mathbb{C}_{(1^2)^3(1)^3}^{\tau} $}
 &\makecell{$\mathbb{C}_{(1^2)^4(1)}^{\tau}   $}   \\
\hline
\makecell{$(8,1)$} & \makecell{$\chi^{(7)}$} &$\cdot$&$\cdot$\\ 
\hline
\makecell{$(7,2)$} & \makecell{$\chi^{(6,1)}$} &$\cdot$&$\cdot$\\ \hline
\makecell{$(7,1^2)$} & \makecell{$\chi^{(7)}\oplus \chi^{(6,1)}$} &$\cdot$ &$\cdot$\\ \hline
\makecell{$(6,3)$} & \makecell{$\chi^{(5,2)}$} & \makecell{$\chi^{(3)} \boxtimes \chi^{(3)}$} &$\cdot$\\ \hline
\makecell{$(6,2,1)$} & \makecell{$\chi^{(6,1)}\oplus \chi^{(5,2)}\oplus \chi^{(5,1^2)}$ }&\makecell{$\chi^{(2,1)} \boxtimes \chi^{(3)}$}&$\cdot$\\ \hline
\makecell{$(6,1^3)$} &\makecell{ $\chi^{(6,1)}\oplus \chi^{(5,1^2)}$} &\makecell{$\chi^{(1^3)} \boxtimes \chi^{(3)}$}& $\cdot$\\ \hline
\makecell{$(5,4)$} & \makecell{$\chi^{(4,3)}$} &  $\chi^{(3)}\boxtimes \chi^{(2,1)}$  &$\chi^{(4)}$\\ \hline
$(5,3,1)$ & $\chi^{(5,2)}\oplus \chi^{(4,3)}\oplus \chi^{(4,2,1)}$ &\makecell{$\substack{ (\chi^{(3)} \boxtimes \chi^{(3)})\oplus ( \chi^{(2,1)} \boxtimes \chi^{(3)}) \\ \oplus (\chi^{(3)} \boxtimes \chi^{(2,1)})\oplus (\chi^{(2,1)} \boxtimes \chi^{(2,1)})}$}& $\chi^{(3,1)}$
\\ \hline
$(5,2^2)$ & $\chi^{(5,1^2)}\oplus \chi^{(4,2,1)}$ &\makecell{$\substack{  ( \chi^{(2,1)} \boxtimes \chi^{(3)}) \\ \oplus (\chi^{(1^3)} \boxtimes \chi^{(3)})\oplus (\chi^{(2,1)} \boxtimes \chi^{(2,1)})}$}&$\chi^{(2,2)}$\\ \hline
$(5,2,1^2)$ & \makecell{$\substack{\chi^{(5,2)}\oplus \chi^{(5,1^2)} \\ \oplus \chi^{(4,2,1)}\oplus \chi^{(4,1^3)}}$} &
\makecell{$\substack{ (\chi^{(3)} \boxtimes \chi^{(3)})\oplus2 ( \chi^{(2,1)} \boxtimes \chi^{(3)}) \\ \oplus (\chi^{(1^3)}\boxtimes \chi^{(3)})\oplus (\chi^{(2,1)} \boxtimes \chi^{(2,1)})\oplus (\chi^{(1^3)} \boxtimes \chi^{(2,1)})}$}
&$\chi^{(2,1^2)}$ \\ \hline
$(5,1^4)$ & $\chi^{(5,1^2)}\oplus \chi^{(4,1^3)}$ &\makecell{$\substack{  ( \chi^{(2,1)} \boxtimes \chi^{(3)}) \oplus (\chi^{(1^3)}\boxtimes \chi^{(3)})\oplus (\chi^{(2,1)} \boxtimes \chi^{(2,1)})}$}&
$\chi^{(1^4)}$ \\ \hline
$(4^2,1)$ & $\chi^{(4,3)}\oplus \chi^{(3^2,1)}$ &\makecell{$\substack{   (\chi^{(3)} \boxtimes \chi^{(2,1)})\oplus (\chi^{(2,1)}\boxtimes \chi^{(2,1)})\oplus(\chi^{(3)}\boxtimes \chi^{(1^3)})}$}&$\chi^{(4)}\oplus \chi^{(3,1)}$
\\ \hline
\makecell{$(4,3,2)$ }& $\chi^{(4,2,1)}\oplus \chi^{(3^2,1)}\oplus \chi^{(3,2^2)}$ &\makecell{$\substack{  (\chi^{(3)}\boxtimes \chi^{(3)})\oplus (\chi^{(2,1)}\boxtimes \chi^{(3)}) \oplus (\chi^{(3)} \boxtimes \chi^{(2,1)}) \\ \oplus 2(\chi^{(2,1)}\boxtimes \chi^{(2,1)})
\oplus(\chi^{(1^3)}\boxtimes \chi^{(2,1)}) \oplus ( \chi^{(2,1)}\boxtimes \chi^{(1^3)})}$}&
$\chi^{(3,1)}\oplus \chi^{(2,2)}\oplus \chi^{(2,1^2)}$\\ \hline
$(4,3,1^2)$ & \makecell{$\substack{ \chi^{(4,3)}\oplus \chi^{(4,2,1)} \\ \oplus \chi^{(3^2,1)}\oplus \chi^{(3,2,1^2)}}$} &
\makecell{$\substack{  (\chi^{(2,1)}\boxtimes \chi^{(3)})\oplus (\chi^{(1^3)}\boxtimes \chi^{(3)})\\ \oplus 2(\chi^{(3)} \boxtimes \chi^{(2,1)})\oplus 3(\chi^{(2,1)}\boxtimes \chi^{(2,1)})\\
\oplus(\chi^{(1^3)}\boxtimes \chi^{(2,1)}) \oplus ( \chi^{(3)}\boxtimes \chi^{(1^3)})\oplus ( \chi^{(2,1)}\boxtimes \chi^{(1^3)})}$}
&
\makecell{$\substack{\chi^{(4)}\oplus 2\chi^{(3,1)} \\ \oplus  \chi^{(2,2)}\oplus \chi^{(2,1^2)}}$ }\\ \hline
$(4,2^2,1)$ & \makecell{$\substack{ \chi^{(4,2,1)}\oplus \chi^{(4,1^3)} \\ \oplus \chi^{(3,2^2)}\oplus \chi^{(3,2,1^2)}}$ } &
\makecell{$\substack{  (\chi^{(3)}\boxtimes \chi^{(3)}) \oplus 2 (\chi^{(2,1)}\boxtimes \chi^{(3)})\\ \oplus (\chi^{(1^3)}\boxtimes \chi^{(3)}) \oplus (\chi^{(3)} \boxtimes \chi^{(2,1)})\oplus 3(\chi^{(2,1)}\boxtimes \chi^{(2,1)})\\
\oplus2(\chi^{(1^3)}\boxtimes \chi^{(2,1)}) \oplus ( \chi^{(2,1)}\boxtimes \chi^{(1^3)})\oplus ( \chi^{(1^3)}\boxtimes \chi^{(1^3)})}$}
&
\makecell{$\substack{\chi^{(3,1)}\oplus \chi^{(2,2)} \\ \oplus  2\chi^{(2,1^2)}\oplus \chi^{(1^4)}}$ }
\\ \hline
$(4,2,1^3)$ & \makecell{ $\substack{ \chi^{(4,2,1)}\oplus \chi^{(4,1^3)} \\ \oplus \chi^{(3,2,1^2)}\oplus \chi^{(3,1^4)}}$} &
\makecell{$\substack{  (\chi^{(3)}\boxtimes \chi^{(3)}) \oplus 2 (\chi^{(2,1)}\boxtimes \chi^{(3)})\\ \oplus (\chi^{(1^3)}\boxtimes \chi^{(3)}) \oplus (\chi^{(3)} \boxtimes \chi^{(2,1)})\oplus 3(\chi^{(2,1)}\boxtimes \chi^{(2,1)})\\
\oplus2(\chi^{(1^3)}\boxtimes \chi^{(2,1)}) \oplus ( \chi^{(2,1)}\boxtimes \chi^{(1^3)})\oplus ( \chi^{(1^3)}\boxtimes \chi^{(1^3)})}$}
& \makecell{$\substack{\chi^{(3,1)}\oplus \chi^{(2,2)} \\ \oplus  2\chi^{(2,1^2)}\oplus \chi^{(1^4)}}$ } \\ \hline
$(4,1^5)$ & \makecell{$\chi^{(4,1^3)}\oplus \chi^{(3,1^4)}$} &
\makecell{$\substack{  (\chi^{(3)}\boxtimes \chi^{(3)}) \oplus  (\chi^{(2,1)}\boxtimes \chi^{(3)})\\ \oplus (\chi^{(2,1)}\boxtimes \chi^{(2,1)}) \oplus (\chi^{(1^3)}\boxtimes \chi^{(2,1)}) \oplus ( \chi^{(1^3)}\boxtimes \chi^{(1^3)})}$}
& $\chi^{(2,1^2)}\oplus \chi^{(1^4)}$ \\ \hline
$(3^3)$ & $\chi^{(3,2^2)}$ &
\makecell{$\substack{  (\chi^{(3)}\boxtimes \chi^{(3)}) \oplus (\chi^{(2,1)}\boxtimes \chi^{(2,1)}) \oplus ( \chi^{(1^3)}\boxtimes \chi^{(1^3)})}$}
& $\chi^{(2,1^2)}$\\ \hline
$(3^2,2,1)$ & \makecell{ $\substack{ \chi^{(3^2,1)}\oplus \chi^{(3,2^2)} \\ \oplus \chi^{(3,2,1^2)}\oplus \chi^{(2^3,1)}}$} &
\makecell{$\substack{   (\chi^{(2,1)}\boxtimes \chi^{(3)}) \oplus 2 (\chi^{(3)} \boxtimes \chi^{(2,1)}) \\ \oplus 3(\chi^{(2,1)}\boxtimes \chi^{(2,1)})
\oplus2(\chi^{(1^3)}\boxtimes \chi^{(2,1)})  \\ \oplus ( \chi^{(3)}\boxtimes \chi^{(1^3)})\oplus 2( \chi^{(2,1)}\boxtimes \chi^{(1^3)})\oplus ( \chi^{(1^3)}\boxtimes \chi^{(1^3)})}$}
&
\makecell{$\substack{\chi^{(4)}\oplus 2\chi^{(3,1)} \\ \oplus \chi^{(2,2)}\oplus  2\chi^{(2,1^2)}\oplus \chi^{(1^4)}}$ } 
\\ \hline
$(3^2,1^3)$ & \makecell{$\chi^{(3^2,1)}\oplus \chi^{(3,2,1^2)}\oplus \chi^{(2^2,1^3)}$} &
\makecell{$\substack{   (\chi^{(1^3)}\boxtimes \chi^{(3)}) \oplus  (\chi^{(3)} \boxtimes \chi^{(2,1)})  \oplus 3(\chi^{(2,1)}\boxtimes \chi^{(2,1)})
\\ \oplus (\chi^{(1^3)}\boxtimes \chi^{(2,1)})   \oplus 2 ( \chi^{(3)}\boxtimes \chi^{(1^3)})\oplus 2( \chi^{(2,1)}\boxtimes \chi^{(1^3)})}$}
 &
 \makecell{$\substack{\chi^{(4)}\oplus 2\chi^{(3,1)} \\ \oplus2 \chi^{(2,2)}\oplus  \chi^{(2,1^2)}}$ } 
 \\ \hline
$(3,2^3)$ & \makecell{$\chi^{(3,2,1^2)}\oplus \chi^{(2^3,1)}$ }&
\makecell{$\substack{   (\chi^{(1^3)}\boxtimes \chi^{(3)}) \oplus  (\chi^{(3)} \boxtimes \chi^{(2,1)}) \\ \oplus 2(\chi^{(2,1)}\boxtimes \chi^{(2,1)})
 \oplus (\chi^{(1^3)}\boxtimes \chi^{(2,1)})   \\
 \oplus  ( \chi^{(3)}\boxtimes \chi^{(1^3)})\oplus ( \chi^{(2,1)}\boxtimes \chi^{(1^3)}) \oplus (\chi^{(1^3)}\boxtimes \chi^{(1^3)})}$}
&
 \makecell{$\substack{\chi^{(4)}\oplus \chi^{(3,1)} \\ \oplus \chi^{(2,2)}\oplus  \chi^{(2,1^2)}\oplus \chi^{(1^4)} } $ } 
 \\ \hline
$(3,2^2,1^2)$ & \makecell{ $\substack{ \chi^{(3,2^2)}\oplus \chi^{(3,2,1^2)} \\ \oplus \chi^{(3,1^4)}\oplus \chi^{(2^3,1)}\oplus \chi^{(2^2,1^3)}}$} & 
\makecell{$\substack{  (\chi^{(3)}\boxtimes \chi^{(3)}\oplus (\chi^{(2,1)}\boxtimes \chi^{(3)}) \\
\oplus  2(\chi^{(3)} \boxtimes  \chi^{(2,1)})  \oplus 4(\chi^{(2,1)}\boxtimes \chi^{(2,1)})
 \oplus 2(\chi^{(1^3)}\boxtimes \chi^{(2,1)})   \\
 \oplus  ( \chi^{(3)}\boxtimes \chi^{(1^3)})\oplus 3( \chi^{(2,1)}\boxtimes \chi^{(1^3)}) \oplus2 (\chi^{(1^3)}\boxtimes \chi^{(1^3)})}$}
   &
    \makecell{$\substack{ 3\chi^{(3,1)}  \oplus \chi^{(2,2)}\\ \oplus  3\chi^{(2,1^2)}\oplus \chi^{(1^4)} } $ } 
    \\ \hline
$(3,1^6)$ & $\chi^{(3,1^4)}\oplus \chi^{(2,1^5)}$ & 
\makecell{$\substack{ (\chi^{(3)}\boxtimes\chi^{(3)} ) \oplus (\chi^{(3)}\boxtimes\chi^{(2,1)} )\\ \oplus (\chi^{(2,1)}\boxtimes\chi^{(2,1)}) \oplus (\chi^{(2,1)}\boxtimes\chi^{(1^3)}) \oplus(  \chi^{(1^3)}\boxtimes\chi^{(1^3)}  )}$}
&$\chi^{(3,1)}\oplus \chi^{(2,1^2)}$\\ \hline
$(2^4,1)$ & $\chi^{(2^3,1)}\oplus \chi^{(2^2,1^3)}$ & 
\makecell{$\substack{ (\chi^{(3)}\boxtimes\chi^{(2,1)} ) \oplus (\chi^{(2,1)}\boxtimes\chi^{(2,1)}) \\ \oplus (\chi^{(1^3)}\boxtimes \chi^{(2,1)}) \oplus  (\chi^{(3)}\boxtimes\chi^{(1^3)}) \oplus 2(  \chi^{(2,1)}\boxtimes\chi^{(1^3)}  )}$}
&\makecell{$\substack{\chi^{(4)}\oplus \chi^{(3,1)}  \oplus \chi^{(2^2)}\oplus \chi^{(2,1^2)}}$}\\ \hline
$(2^3,1^3)$ & \makecell{ $\chi^{(2^3,1)}\oplus \chi^{(2^2,1^3)} \oplus \chi^{(2,1^5)}$} &
\makecell{$\substack{ (\chi^{(3)}\boxtimes\chi^{(2,1)} ) \oplus 2 (\chi^{(2,1)}\boxtimes\chi^{(2,1)}) \\ \oplus  2 (\chi^{(3)}\boxtimes\chi^{(1^3)}) \oplus 2(  \chi^{(2,1)}\boxtimes\chi^{(1^3)}  ) \oplus ( \chi^{(1^3)}\boxtimes \chi^{(1^3)})}$}
 &\makecell{$\substack{\chi^{(4)}\oplus 2\chi^{(3,1)}  \oplus \chi^{(2^2)}\oplus \chi^{(2,1^2)}}$}\\ \hline
$(2^2,1^5)$ & \makecell{ $\chi^{(2^2,1^3)}\oplus \chi^{(2,1^5)} \oplus \chi^{(1^7)}$} &\makecell{$\substack{ (\chi^{(3)}\boxtimes\chi^{(2,1)} ) \oplus  (\chi^{(2,1)}\boxtimes\chi^{(2,1)}) \\ \oplus  2 (\chi^{(3)}\boxtimes\chi^{(1^3)}) \oplus 2(  \chi^{(2,1)}\boxtimes\chi^{(1^3)}  ) }$}
 &\makecell{$\chi^{(4)}\oplus 2\chi^{(3,1)}  \oplus \chi^{(2^2)} $}\\ \hline
$(2,1^7)$ & \makecell{ $\chi^{(2,1^5)}  \oplus \chi^{(1^7)}$} & 
\makecell{$\substack{ (\chi^{(3)}\boxtimes\chi^{(2,1)} )  \oplus  (\chi^{(3)}\boxtimes\chi^{(1^3)}) \oplus (  \chi^{(2,1)}\boxtimes\chi^{(1^3)}  ) }$}
& $\chi^{(4)}\oplus \chi^{(3,1)}$ \\ \hline
$(1^9)$ & $\chi^{(1^7)}$ &
\makecell{$\chi^{(3)}\boxtimes \chi^{(1^3)}$}& $\chi^{(4)}$ \\ \hline
\end{tabular}
\end{figure}

%\begin{figure}[H]
%\centering
%\begin{tabular}{|cccc|}

%\end{tabular}
%\end{figure}

To compute the multiplicity of $1$ for instance in 

\begin{align*}
\C_{((1^2),(1^2),(1^2))((1),(1),(1))^7}^{(\mu,\mu,(5,1^4))}&=\mathbb{C}^\mu_{(1^2)(1)^7}\otimes\mathbb{C}^\mu_{(1^2)(1)^7}\otimes\mathbb{C}^{(5,1^4)}_{(1^2)(1)^7}
\\
&=\chi^{(4,3)}\otimes\chi^{(4,3)}\otimes\left(\chi^{(5,1^2)}\oplus\chi^{(4,1^3)}\right)
\end{align*}
we use the tables in \cite{Gi} to see that the trivial character appears with multiplicity $1$ in both 

$$
\chi^{(4,3)}\otimes\chi^{(4,3)}\otimes\chi^{(5,1^2)}\hspace{.5cm}\text{ and }\hspace{.5cm}\chi^{(4,3)}\otimes\chi^{(4,3)}\otimes\chi^{(4,1^3)},
$$
from which we get $M^{(5,1^4)}_{((1^2),(1^2),(1^2))((1),(1),(1))^7}=2$.

\end{document}